\documentclass[11pt,letter,reqno]{amsart}

\usepackage{amsmath,amsfonts,amscd,flafter,epsf,amssymb,epsfig,bm,url, mathrsfs, amsthm, subfig}
\usepackage{enumitem}
\usepackage{makecell}
\usepackage[margin=1in]{geometry}

\theoremstyle{plain}

\newtheorem{question}{Question}
\newtheorem{thm}{Theorem}[section]
\newtheorem{lemma}[thm]{Lemma}
\newtheorem{cor}[thm]{Corollary}
\newtheorem{prop}[thm]{Proposition}

\newtheorem{claim}[thm]{Claim}

\theoremstyle{definition}
\newtheorem{example}[thm]{Example}
\newtheorem{mydef}[thm]{Definition}

\theoremstyle{remark}

\usepackage{verbatim}

\newcommand{\bbd}{\mathbb{D}}

\newcommand{\dps}[1]{\displaystyle{#1}}

\newcommand{\class}[1]{\langle #1 \rangle}
\newcommand{\gauss}[1]{\left[ #1 \right]}
\newcommand{\bracket}[1]{\left(#1\right)}

\DeclareMathOperator{\arccosh}{arccosh}

\DeclareMathOperator{\sys}{sys}
\DeclareMathOperator{\kiss}{kiss}
\DeclareMathOperator{\area}{Area}
\DeclareMathOperator{\perim}{Perim}

\let\dd\undefined
\makeatletter \newcommand\dd[1]{\ensuremath{%
  \mathop{}\!\mathrm{d}#1\@ifnextchar\dd{\!}{}}}
\makeatother

\usepackage{hyperref}
\hypersetup{colorlinks = true, allcolors = blue}

\begin{document}

\title{Shortest filling geodesics on hyperbolic surfaces}

\author{Yue Gao}
\address{School of Mathematics and Statistics, Anhui Normal University, Wuhu, Anhui, 241002, P. R. China}
\email{yuegao@ahnu.edu.cn}

\author{Jiajun Wang}
\address{LMAM, School of Mathematical Sciences, Peking University, Beijing, 100871, P. R. China}
\email{wjiajun@pku.edu.cn}

\author{Zhongzi Wang}
\address{LMAM, School of Mathematical Sciences, Peking University, Beijing, 100871, P. R. China}
\email{wangzz22@stu.pku.edu.cn}

\maketitle

\begin{abstract}
In this paper, we obtain the minimal length of a filling (multi-)geodesic on a genus $g$ hyperbolic surface in the moduli space of hyperbolic surfaces and show that it is realized by the geodesic whose complement is a right-angled regular $(8g-4)$-gon. A single geodesic realizing this minimum is provided. 
\end{abstract}

\setcounter{tocdepth}{1}
\tableofcontents

\section{Introduction}
A set of closed geodesics on a hyperbolic surface is filling if the geodesics cut the surface into disks. The study on filling geodesics on hyperbolic surfaces is initiated by Thurston in a series of work, including Nielsen realization Theorem \cite{kerckhoff1983nielsen}; construction of pseudo-Anosov maps \cite{thurston1988geometry}; Thurston's spine and his conjecture on the deformation retract of the moduli space of all genus g hyperbolic surfaces \cite{thurston1986spine}. This concept is widely concerned in the study of hyperbolic surfaces and various topics, for example, the mapping class group and curve complex \cite{masur1999geometry,masur2000geometry}, the systole function \cite{schaller1999systoles} and the spectra of Laplacian of hyperbolic surfaces \cite{wu2022random, lipnowski2024towards, monk2025random}. Filling geodesics are not rare, in fact, on a random hyperbolic surface, most sufficiently long geodesics are filling  \cite{wu2022prime, dozier2023counting}. A set of closed geodesics is called a \emph{multi-geodesic} and the length of a multi-geodesic is the total length of the geodesics.

A natural question is:

\begin{question}\label{ques:start}
What is the minimal length of a filling multi-geodesic on a hyperbolic surface of genus $g$, among the moduli space of all genus g hyperbolic surfaces?
\end{question}

The minimal length can be realized by Mumford's compactness theorem \cite{mumford1971remark} and the collar lemma. 
Our main result answers Question \ref{ques:start}.

\begin{thm}\label{thm:main}
The length of a filling multi-geodesic on a closed orientable hyperbolic surface of genus $g$ is no less than half of the perimeter of the regular right-angled hyperbolic $(8g - 4)$-gon, and the minimal length can be realized by a single filling geodesic.
\end{thm}

Let $P_k$ ($k\geqslant 5$) be the perimeter of the regular right-angled (hyperbolic) $k$-gon. We give the criterion for a filling multi-geodesic to have the minimal length.

\begin{thm}\label{thm:rigid}
A filling multi-geodesic $\Gamma$ on a closed orientable hyperbolic surface of genus $g$ has length $\frac12P_{8g-4}$ if and only if $X\setminus\Gamma$ is a regular right-angled $(8g-4)$-gon. 
\end{thm}

A \emph{filling pair} is a pair of simple closed geodesics that fills the surface and the minimal length of filling pairs was previously studied. Aougab-Huang \cite{aougab2015minimally} conjectured that any filling pair on a hyperbolic surface has a length at least $\frac12P_{8g-4}$, and they proved the conjecture for the case that the complement of the filling pair has at most two components. Sanki-Vadnere \cite{sanki2021conjecture} and Gaster \cite{gaster2021short} proved Aougab-Huang's conjecture for general filling pairs independently. The proof by Sanki-Vadnere relies on an isoperimetric inequality for polygons with even edges.

Besides a construction of a single filling geodesic realizing the minimal length, there are two ingredients in the proof of Theorem \ref{thm:main}. A filling multi-geodesic is considered as a geodesic graph on the surface. We first construct a filling graph with a smaller length so that every component of its complement has at least 5 edges (Theorem \ref{thm:filling_graph}). Then we apply an isoperimetric inequality (Theorem \ref{thm:isoperim}) to the complement to obtain the lower bound. 

Theorem \ref{thm:isoperim} generalizes the isoperimetric inequality by Sanki-Vadnere \cite{sanki2021conjecture} to polygons with at least 4 edges. For filling pairs, there are only double intersection points and every complementary disk has even edges. Thus Sanki-Vadnere's isoperimetric inequality suffices to prove the length lower bound for filling pairs. Theorem \ref{thm:main} considers arbitrary filling multi-geodesics, for which multiple intersection points and triangles are inevitable. It is necessary to generalize the isoperimetric inequality by Sanki-Vadnere to arbitrary polygons with at least four edges (which is the best possible case, see Example \ref{emp:triangle_fail}). The method to prove Theorem \ref{thm:isoperim} is influenced by Sanki-Vadnere's work, but technically more difficult. Since the isoperimetric inequality may fail when there are triangles, it is crucial to find such a triangle-free graph as in Theorem \ref{thm:filling_graph}. The proof of Theorem \ref{thm:filling_graph} is inspired by the hierarchy argument in 2- and 3-dimensional topology.

A natural corollary to Theorem \ref{thm:main} on the systole of hyperbolic surfaces is 

\begin{cor}\label{cor:sys}
Let $X$ be a genus $g$ closed orientable hyperbolic surface. Let $\sys(X)$ be the systole length for $X$, that is, the minimal length of simple closed geodesics on $X$. Let $\Gamma$ be the set of simple closed geodesics with length $\sys(X)$ and $\kiss(X)=\#\Gamma$ be the kissing number of $X$. If $\Gamma$ fills $X$, then 
\begin{equation*}
\sys(X) \kiss(X) \geqslant \frac{1}{2} P_{8g-4}. 
\end{equation*}
\end{cor}

It is proved in \cite{bourque2023linear} that $\sys(X) \leqslant 2\log g + 2.409$ when $g$ is sufficiently large. Since $P_{8g-4}=2(8g-4)\arccosh(\sqrt{2}\cos(\pi/(8g-4)) )$, we have

\begin{cor}\label{cor:kiss}
When $g$ is sufficiently large, for a genus $g$ hyperbolic surface $X$ with filling systole, 
\begin{equation}\label{eqn:kissing_number_bound}
\kiss(X) \geqslant \frac{ \frac{1}{2} P_{8g-4}}{\sys(X)} \geqslant  3.525 \frac{g}{\log g}. 
\end{equation}
\end{cor}

Corollary \ref{cor:kiss} improves the best known lower bound $\pi \frac{g}{\log g}$ by Anderson-Parlier-Pettet \cite{anderson2011}. 

For other constructions of a filling geodesic with $2g-1$ intersections, see for example \cite{arettines2015geometry} and his construction relies on ribbon graph. 
For filling pairs, when $g>2$, several examples with $2g-1$ intersections are constructed by Aougab-Huang \cite{aougab2015minimally}, Chang-Menasco \cite{chang2023construction} and Nieland \cite{nieland2016connectedsum}. 

If the filling condition is removed, a non-simple geodesic with self-intersection number $k$ has a length at least $C\sqrt{k}$ \cite{basmajian2013universal}. The minimal length of non-simple geodesics behaves quite differently to filling geodesics, which has been widely studied, see 
\cite{yamada1982marden, hempel1984traces, basmajian1993stable, baribaud1999closed, basmajian2013universal, basmajian2024shortest, erlandsson2020short, shen2022minimal, shen2023nonsimple, vo2022short, wang2024length}.

The minimal length of filling multi-geodesics can be regarded as a function on the moduli space $\mathcal{M}_g$, which has some similarity with the systole function. When approaching the Deligne-Mumford boundary of the moduli space, the systole function approaches to $0$ while the minimal length of filling multi-geodesics approaches to $\infty$. The maximum of the systole function has many estimates (see for example \cite{brooks1988injectivity, buser1994period, katz2007logarithmic, liu2023random, katz2025logarithmic, bourque2023linear}) and only been calculated in genus $2$ \cite{jenni1984ersten, bavard1992systole, schaller1999systoles}.

The paper is organized as follows. In Section \ref{sec:subgraph}, we show that, for any filling multi-geodesic, there exists a filling geodesic graph with a smaller length so that every component of its complement is a polygon with at least five edges. In Section \ref{sec:iso}, the isoperimetric inequality for polygons with at least four edges is established. In Section \ref{sec:minimal}, we construct a filling geodesic for any genus $g\geqslant2$ whose complement is the regular right-angled $(8g-4)$-gon, and prove Theorem \ref{thm:main}.

\subsection*{Acknowledgement:} 
We would like to thank Professors Hugo Parlier, Yunhui Wu and Ying Zhang for their helpful discussions. 

The first named author is supported by NSFC grant 12301082. The second named author is partially supported by the National Key R\&D Program of China 2020YFA0712800 and NSFC grant 12131009.

\section{Filling graph from filling geodesics}\label{sec:subgraph}

We prove that for any filling multi-geodesic, we can find a filling triangle-free geodesic graph.

\begin{thm}\label{thm:filling_graph}
Let $X$ be a closed hyperbolic surface of genus $g$ and $\Gamma$ be a filling multi-geodesic. Then there exists a geodesic graph $G\subset X$ such that $X\setminus G$ consists of polygons $D_1, D_2,\cdots, D_k$ with edges numbers $m_1,m_2,\cdots,m_k$ respectively, satisfying the following 
\begin{equation}\label{eqn:subgraph_length_decreasing}
\ell(G)\leqslant\ell(\Gamma),
\end{equation}
\begin{equation}\label{eqn:subgraph_triangle_free}
m_i\geqslant 5,\quad i=1,2,\cdots,k,
\end{equation}
\and 
\begin{equation}\label{eqn:subgraph_complexity}
m_1+\cdots+m_k=8g-8+4k.
\end{equation}
\end{thm}

When $\Gamma$ has only double intersection points, $G$ can be chosen to be a subgraph of $\Gamma$. When $\Gamma$ has multiple intersection points, $G$ is in general a modification of a subgraph of $\Gamma$ to satisfy \eqref{eqn:subgraph_complexity}. The proof of Theorem \ref{thm:main} only requires $m_i\geqslant4$.

\subsection{Geodesic graphs and convex surfaces}

We prove some preliminary results on geodesic graphs and hyperbolic surfaces with piecewise geodesic boundary.

A graph is \emph{finite} if it has finite vertices and edges. All graphs in the paper are finite. An embedded graph $G$ on $X$ is \emph{geodesic} if all edges are geodesic. The \emph{length} of a geodesic graph $G$ is the sum of the lengths of its edges, denoted by $\ell(G)$.

For a nonempty geodesic graph $G$ on $X$, $X\setminus G$ is a hyperbolic surface with piecewise geodesic boundary (we always neglect the vertices on the boundary with straight angles).

\begin{mydef}\label{def:cvx}
A hyperbolic surface with piecewise geodesic boundary is \emph{convex} if the inner angle of every vertex on its boundary is smaller than $\pi$. A nonempty geodesic graph $G$ on $X$ is \emph{convex} if all vertices of $G$ are 3- or 4-valent, and, as illustrated in Figure \ref{fig:subgraph_convexity}, $G$ satisfies
\begin{figure}[htbp]
\centering
\includegraphics[width=.5\textwidth]{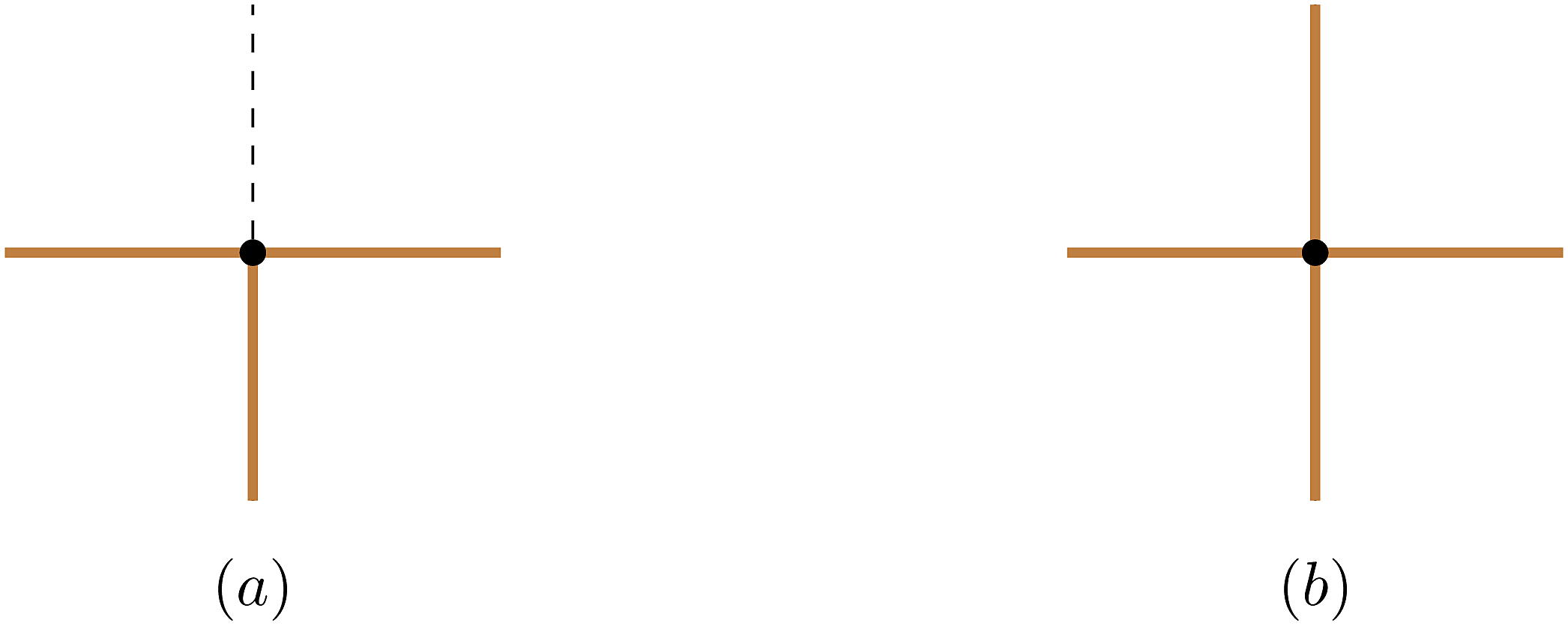}
\caption{Vertices of convex graphs}\label{fig:subgraph_convexity}
\end{figure}
\begin{enumerate}
\item every trivalent vertex contains two opposite edges;
\item every 4-valent vertex is the intersection of two geodesic segments.
\end{enumerate}
A convex graph $G$ is \emph{filling} if every connected component of $X\setminus G$ is a polygon. If $G\subset X$ is convex, then $X\setminus G$ is a convex surface.
\end{mydef}

\begin{lemma}\label{lemma:noncontractible_component_geodesic}
Let $X$ be a closed hyperbolic surface of genus $g$ and $G$ be a nonempty geodesic graph such that $X\setminus G$ is a convex hyperbolic surface. If a component $S$ of $X\setminus G$ is not contractible, then any boundary component $\alpha$ of $S$ cobounds an annulus with a simple closed geodesic $\gamma$ in $X$ such that the annulus $A$ has inner angles less than $\pi$ for all vertices on $\alpha$. (It follows that $\gamma\subset S$.)
\end{lemma}

\begin{proof}
Let $\alpha$ be a boundary component of $S$. Suppose that $\alpha$ bounds a disk $D$ on $X$. We cannot have $D\subset S$ since $S$ is not contractible. If $D$ lives outside $S$, then, as illustrated in Figure \ref{fig:disk_outside_S}, 
\begin{figure}[htbp]
\centering
\includegraphics[width=.25\textwidth]{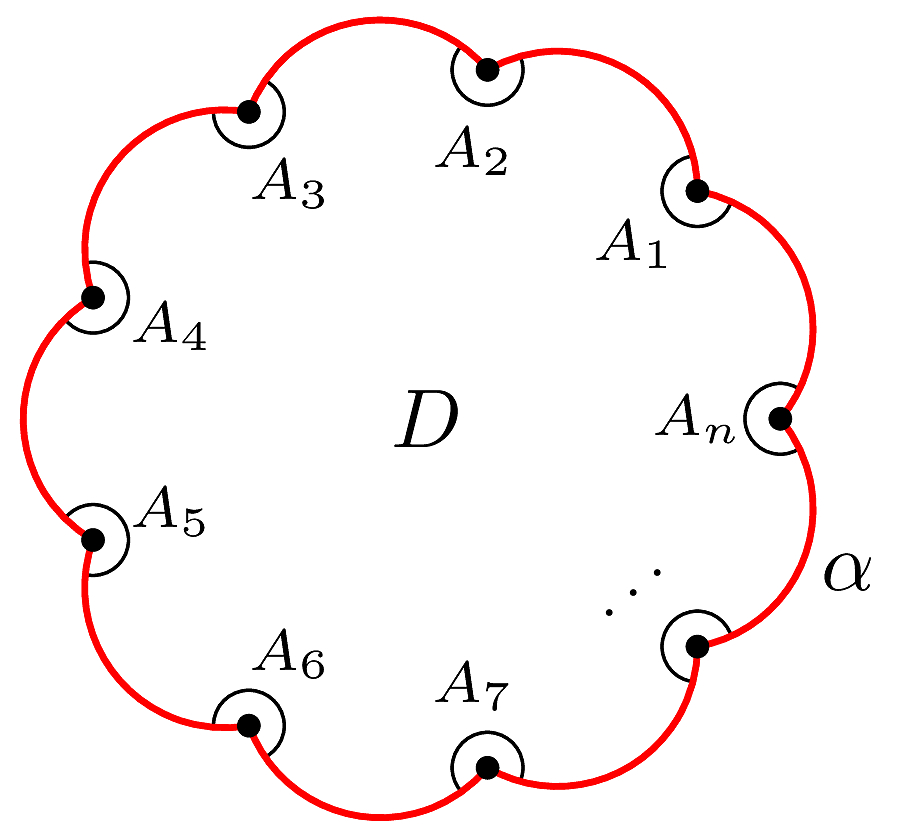}
\caption{A boundary component of $S$ bounding a disk outside $S$.}
\label{fig:disk_outside_S}
\end{figure}
the boundary $\alpha$ of $D$ is piecewise geodesic and has all inner angle $A_i>\pi$ and by the Gauss-Bonnet theorem, $\area(D)=\sum_{i=1}^n(\pi-A_i)-2\pi<0$, which is a contradiction. Hence $\alpha$ is essential in $X$.

Let $\gamma$ be the unique closed geodesic in $X$ in the homotopy class of $\alpha$. Suppose that $\alpha\cap\gamma\neq\emptyset$. For the universal covering $p:\bbd\to X$, as illustrated in Figure \ref{fig:geodesic_intersecting_boundary},
\begin{figure}[htbp]
\centering
\includegraphics[width=.8\textwidth]{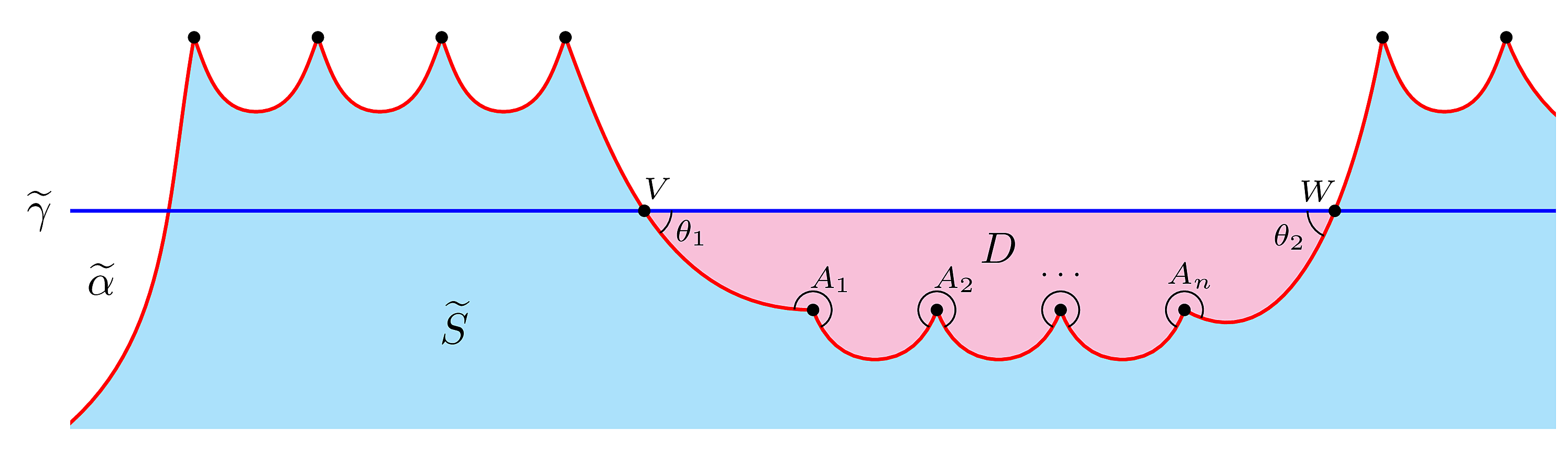}
\caption{$\widetilde{\gamma}$ intersects $\widetilde{\alpha}$}
\label{fig:geodesic_intersecting_boundary}
\end{figure}
a lift $\widetilde{\alpha}\subset\bbd$ of $\alpha$ intersects certain lift $\widetilde{\gamma}$ of $\gamma$. Let $\widetilde{S}$ be a lift of $S$ with $\widetilde{\alpha}$ as a boundary component. 
Since $S$ is convex, $\widetilde{S}$ has inner angles less than $\pi$ for vertices on $\widetilde{\alpha}$.   The deck transformation $g$ corresponding to $\class{\alpha}\in\pi_1(\Sigma)$ gives infinitely many intersection points of $\widetilde\alpha$ and $\widetilde\gamma$. Some adjacent two intersection points gives a polygon $D$ with $2$ vertices $V$ and $W$ on $\widetilde\gamma$, and $n$ vertices on $\widetilde{\alpha}$ (other than $V$ and $W$) with inner angles $A_1,A_2,\cdots,A_n>\pi$. $n$ cannot be zero since any two geodesics do not bound a bigon. Let $\theta_1$ and $\theta_2$ be the inner angles at $V$ and $W$ respectively. By the Gauss-Bonnett theorem, we have $\area(D)=(\pi-\theta_1)+(\pi-\theta_2)+\sum_{i=1}^n(\pi-A_i)-2\pi<0$. Hence $\alpha$ and $\gamma$ are disjoint.

Since $\alpha$ and $\gamma$ are homotopic and disjoint, they cobound an annulus $A$. As in Figure \ref{fig:annulus_alpha_gamma},
\begin{figure}[htbp]
\centering
\includegraphics[width=.6\textwidth]{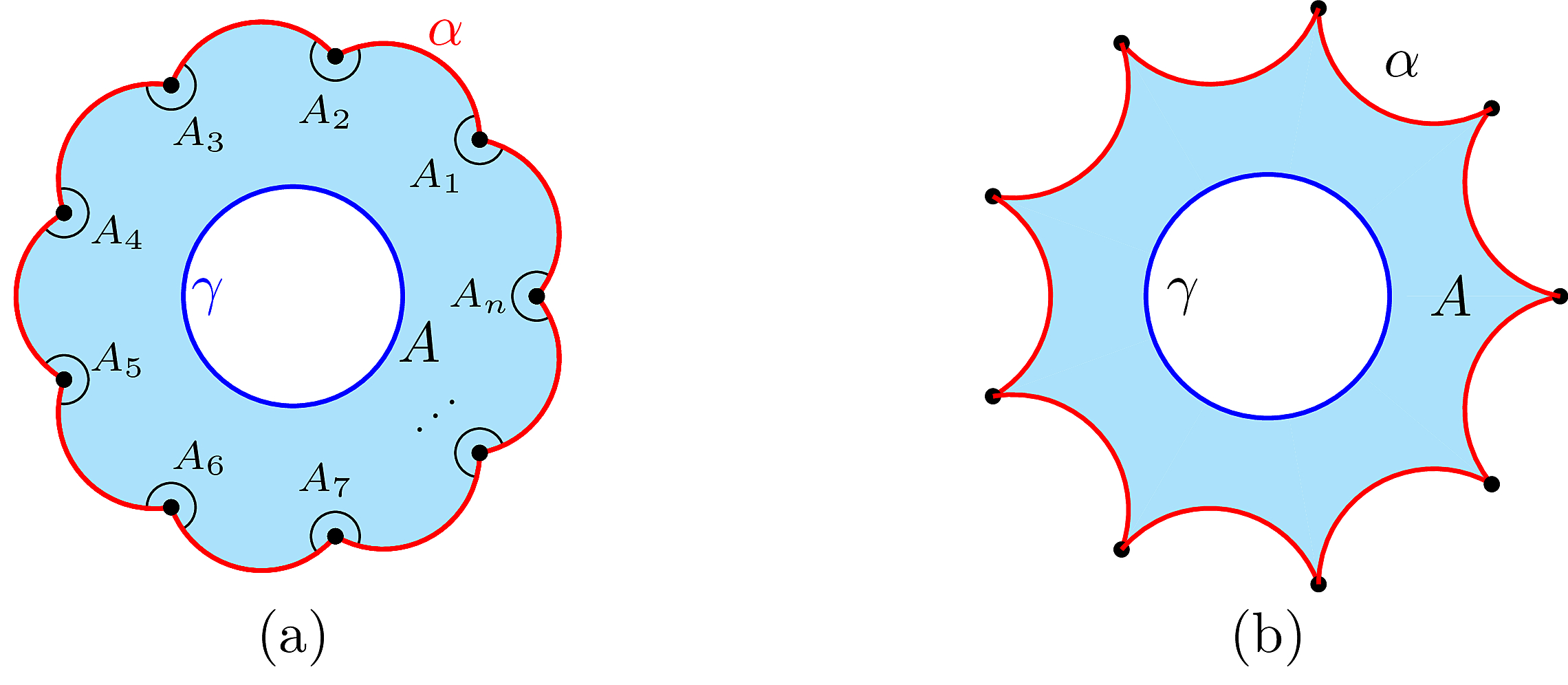}
\caption{An annulus between $\alpha$ and $\gamma$. Note that $\alpha$ does not bound a disk.}
\label{fig:annulus_alpha_gamma}
\end{figure}
either $A$ has inner angles bigger than $\pi$ for all vertices on $\alpha$ (Figure \ref{fig:annulus_alpha_gamma}(a)), or $A$ has inner angles less than $\pi$ for all vertices on $\alpha$ (Figure \ref{fig:annulus_alpha_gamma}(b)). In the first case, the Gauss-Bonnet Theorem implies that $\area(A) = \sum_{i=1}^n(\pi-\alpha_i)<0$. We get a contradiction since $\area(A) > 0$ and $\pi-\alpha_i < 0$ for $i = 1, 2, ...,n$. It follows that $A$ has inner angles (as Figure \ref{fig:annulus_alpha_gamma}(b)) and hence $\gamma\subset S$.
\end{proof}

\subsection{A simple case}
In this subsection, we prove Theorem \ref{thm:filling_graph} in the \emph{simple case} that $\Gamma$ consists of simple closed geodesics and $\Gamma$ has only double intersection points. 

\begin{proof}[Proof of Theorem \ref{thm:filling_graph} in the simple case]

$\Gamma$ is a 4-valent graph on $X$. Take a simple closed geodesic $c\in\Gamma$ and let $G_1=c$. $G_1$ is clearly convex since $X\setminus G_1$ has no vertices.

In $\Gamma\setminus G_1$, there are proper arcs with ends on $G_1$ since otherwise $\Gamma$ cannot be filling. Take an arc $s_1$ in $\Gamma\setminus G_1$ and let $G_2=G_1\cup s_1$. $G_2$ has a single vertex as Figure \ref{fig:subgraph_convexity}(b) if the two ends of $s_1$ coincide, and has two vertices as Figure \ref{fig:subgraph_convexity}(a) otherwise. Hence $G_2$ is convex. 

If $G_2$ is filling, let $G=G_2$ and the expected filling graph is constructed. Otherwise $G_2$ is convex and not filling. Then $X\setminus G_2$ has a non-contractible component $S$ and by Lemma \ref{lemma:noncontractible_component_geodesic}, there is a simple closed geodesic $\gamma$ cobounding an annulus $A$ with some boundary component $\alpha$ of $S$. Since $\Gamma$ is filling, there exists some geodesic $\beta\subset\Gamma$ intersecting $\gamma$ transversely. $\beta$ cannot bound a bigon in $A$ with $\gamma$, and $\beta$ has to intersect $\alpha$. Hence $\Gamma\setminus G_2$ has an arc $s_2\subset\beta$ with both ends on $G_2$. $s_2$ is an embedded arc viewed in $X\setminus G_2$. As in Figure \ref{fig:s2_not_boundary_parallel}, 
\begin{figure}[htbp]
\centering
\includegraphics[width=.55\textwidth]{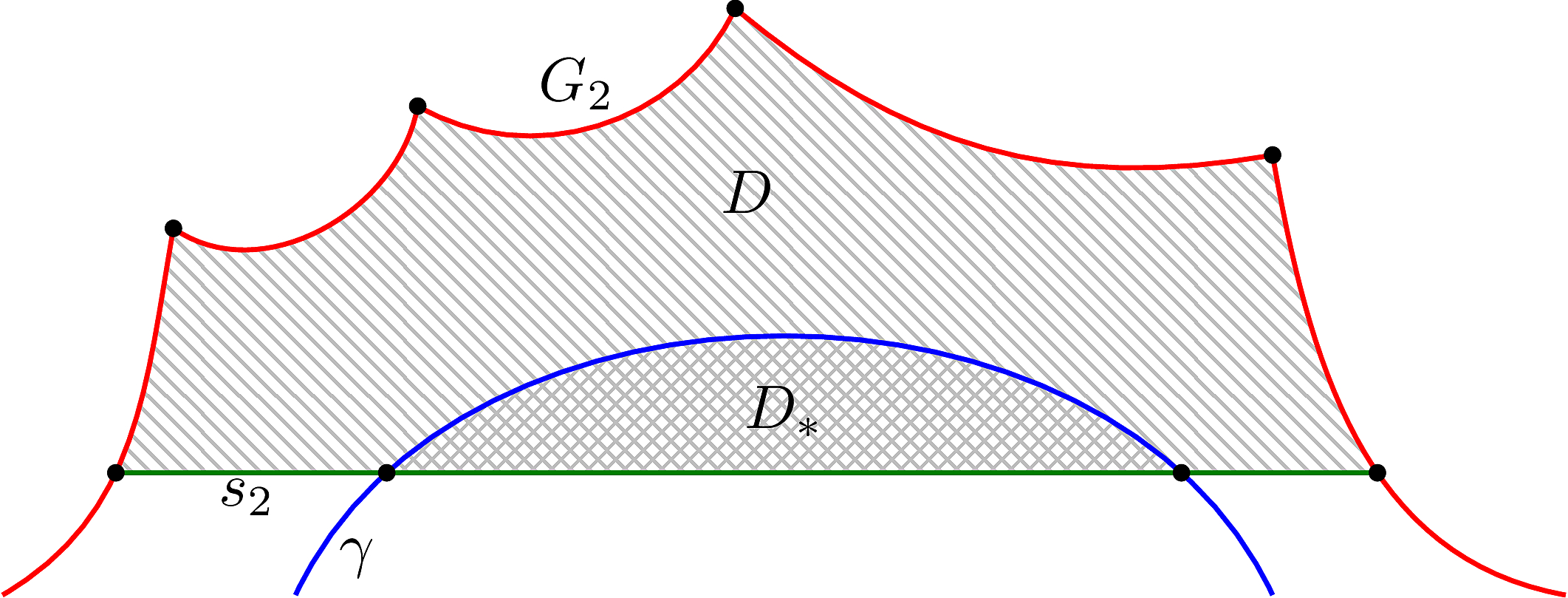}
\caption{If $s_2$ and $G_2$ bound the disk $D$, then $s_2$ and $\gamma$ will bound the bigon $D_*$ since $\gamma$ and $G_2$ are disjoint, where $D_*$ is part of $D$.}
\label{fig:s2_not_boundary_parallel}
\end{figure}
if $s_2$ cobounds a disk $D$ with $G_2$, then $D$ is embedded since $s_2$ is simple, and $s_2$ and $\gamma$ cobound a bigon $D_*$ which is impossible. Hence $s_2$ is \emph{essential} in the sense that it does not co-bound a disk with $G_2$. Let $G_3=G_2\cup s_2$. Since $G_2$ is convex and $\Gamma$ is $4$-valent, the new vertices of $G_3$ are either $3$-valent or $4$-valent and hence $G_3$ is convex.

If $G_3$ is filling, let $G=G_3$ and the expected filling graph is constructed; 
otherwise we repeat the above argument for $G_3$ to get $G_4$, and etc. Since $\Gamma$ is a finite filling graph, the repeating process will eventually stop to get a filling graph $G_n$. Let $G=G_n$. $X\setminus G$ consists of olygons $D_1, D_2, ..., D_k$.

It remains to show $G$ satisfies \eqref{eqn:subgraph_length_decreasing}, \eqref{eqn:subgraph_triangle_free} and \eqref{eqn:subgraph_complexity}. \eqref{eqn:subgraph_length_decreasing} holds since $G$ is a subgraph of $\Gamma$.

The filling graph $G$ is obtained from a simple closed geodesic by adding essential proper arcs one by one. Since the proper arc is essential, a polygon can be generated only when the essential arc cuts an annulus to get a polygon. The annulus cannot be bounded by two simple closed geodesics and it has at least 1 vertex. The essential arc produces $4$ new vertices and the resulting polygon has at least 5 vertices. Therefore every polygon $D_i$ has at least $5$ edges and \eqref{eqn:subgraph_triangle_free} holds. 

When we add an arc to $G$, we get two pairs of complementary inner angles, it follows that the average of the inner angles of $D_1, D_2, ..., D_k$ is $\frac\pi2$. Suppose that $D_1,\cdots,D_k$ are $m_1$-,...,$m_k$-gons respectively and the inner angles of $D_i$ are $A^i_1,\cdots,A^i_{m_i}$ ($i=1,\cdots,k$). Then the area of $D_i$ is $\area(D_i)=\sum_{j=1}^{m_i}(\pi-A^i_j)-2\pi$. 
Since the average of the inner angles is $\pi/2$, we have
$$4\pi(g-1)=\area(X)=\sum_{k=1}^k\area(D_i)=\sum_{i=1}^k\sum_{j=1}^{m_i}(\pi-A^i_j)-2k\pi=\sum_{i=1}^k \bracket{\frac{m_i\pi}2-2\pi}$$
and \eqref{eqn:subgraph_complexity} follows.
\end{proof}

\subsection{General case}

For general filling multi-geodesics, there can be multiple intersection points and the geodesics may have self-intersections. A filling subgraph of $\Gamma$ may fail for \eqref{eqn:subgraph_triangle_free} or \eqref{eqn:subgraph_complexity} and the graph $\Gamma$ need to be modified. The following notion of \emph{cutting curves} plays the same role as the beginning simple closed geodesic and essential arcs in the simple case.

Let $X$ be a closed hyperbolic surface and let $\Theta$ be a filling graph and $G$ be a connected subgraph of $\Theta$ which is possibly empty. When $G=\emptyset$ is empty, a \emph{cutting curve} for the pair $(\Theta, G)$ is a geodesic $\beta:[a,b]\to X$ in $\Theta$ such that
\begin{enumerate}
\item $\beta$ is a simple closed curve; or
\item $\beta$ is injective on $(a,b)$, but not injective on neither $(a,b]$ nor $[a,b)$.
\end{enumerate}
There are exactly four types of cutting curves for $(\Theta,\emptyset)$ as in Figure \ref{fig:cutting_curves_empty_subgraph}.
\begin{figure}[htbp]
\centering
\includegraphics[width=.9\textwidth]{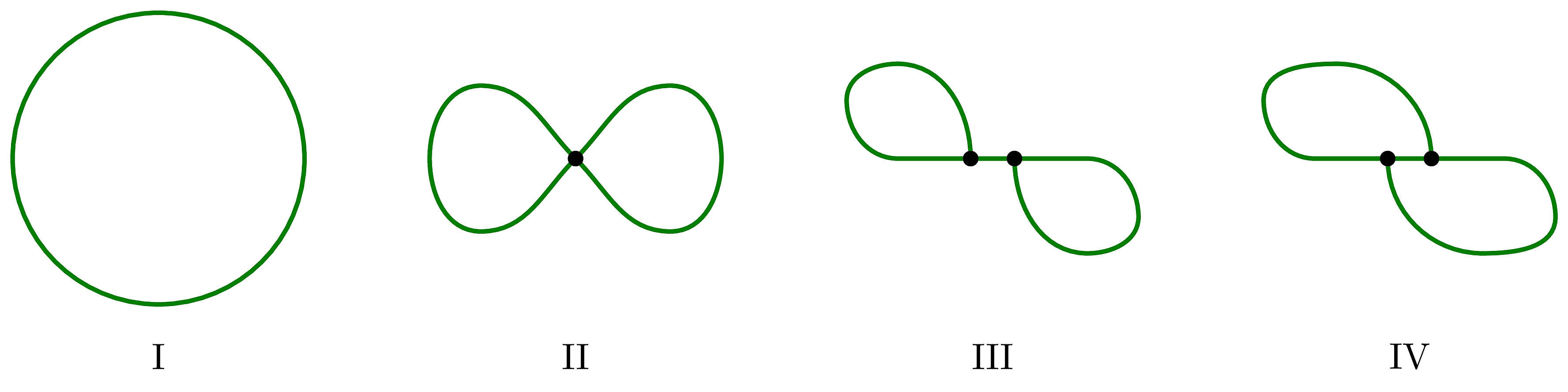}
\caption{Cutting curves for an empty subgraph.}
\label{fig:cutting_curves_empty_subgraph}
\end{figure}
When $G$ is nonempty, a \emph{cutting curve} for $(\Theta,G)$ is a geodesic $\beta:[a,b]\to X$ in $\Theta$ such that $\beta:(a,b)\to X\setminus G$ is injective, $\beta(a)\in G$ and $\beta(b)\in G\cup \beta((a,b))$. There are exactly two types of cutting curves when $G$ is nonempty as in Figure \ref{fig:cutting_curves_nonempty_subgraph}.
\begin{figure}[htbp]
\centering
\includegraphics[width=.7\textwidth]{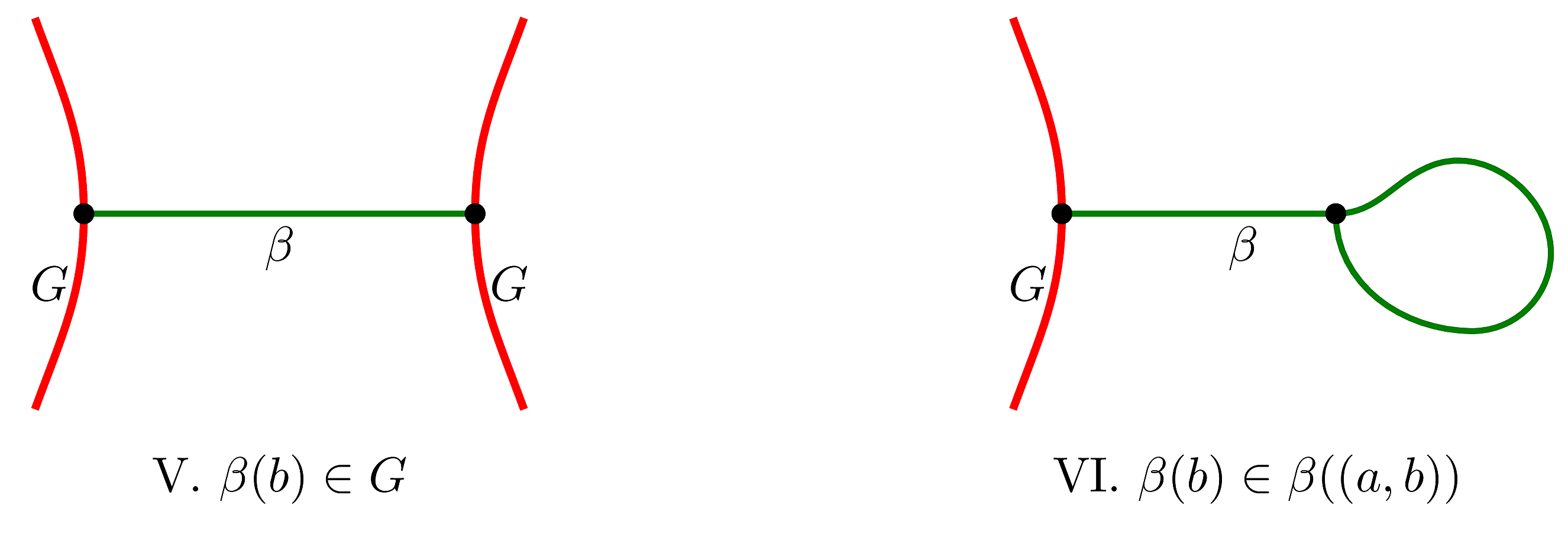}
\caption{Cutting curves for a nonempty subgraph.}
\label{fig:cutting_curves_nonempty_subgraph}
\end{figure}
By abuse of notation, we do not distinguish a cutting curve from its image. A cutting curve for $(\Theta, G)$ is \emph{$\partial$-parallel} if cobounds a disk with $G$, and \emph{essential} otherwise. A $\partial$-parallel cutting curve of type VI can be found in Figure \ref{figure:s2_rho_boundary_parallel_A}.

\begin{proof}[Proof of Theorem \ref{thm:filling_graph}]
We regard $\Gamma$ as a graph on $X$ and the intersection points of $\Gamma$ as vertices. $\Gamma$ is not $4$-valent in general. Let $\Theta_1=\Gamma$ and $G_1=\emptyset$.

Pick any open embedded geodesic segment $\beta:(s,t)\to X$ of $\Gamma$, and extend it along the increasing direction until either we get a simple closed curve, or we reach some point on $\beta$. In the latter case, we extend the curve in the decreasing direction until we reach some point on $\beta$. In both cases, we get a cutting curve $\beta_1:[a,b]\to X$ for $(\Theta_1,G_1)$. Let $\Theta_2=\Theta_1$ and $G_2=\beta_1$.  $G_2$ is convex. 

\begin{claim}\label{claim:no_essential_cutting_implies_filling}
If $G_2$ is not filling, then $(\Theta_2,G_2)$ has a non-$\partial$-parallel proper geodesic arc.
\end{claim}

Since $G_2$ is not filling, $X\setminus G_2$ has a non-contractible component $S$ and by Lemma \ref{lemma:noncontractible_component_geodesic}, there is a simple closed geodesic $\gamma$ disjoint from $G_2$, such that $\gamma$ and $G_2$ cobound an annulus $A$ and $A$ has inner angles $<\pi$ for all vertices. Let $\alpha$ be the other boundary of $A$. Since $\Theta_2$ is filling, there exists some geodesic $\beta$ in $X\setminus G_2$ intersecting $\gamma$ transversely. $\beta$ cannot bound a bigon in $A$ with $\gamma$, and $\beta$ has to intersect $G_2$. Hence $\beta$ cannot be a closed geodesic and $\beta$ is a geodesic arc in $\Gamma\setminus G_2$ with both ends on $G_2$. $\beta$ is in general an immersed arc. As in Figure \ref{fig:s2_not_boundary_parallel} (the embedded case) and \ref{fig:s2_not_boundary_parallel_2} (the immersed case),
\begin{figure}[htbp]
\centering
\subfloat{\includegraphics[width=0.32\textwidth]{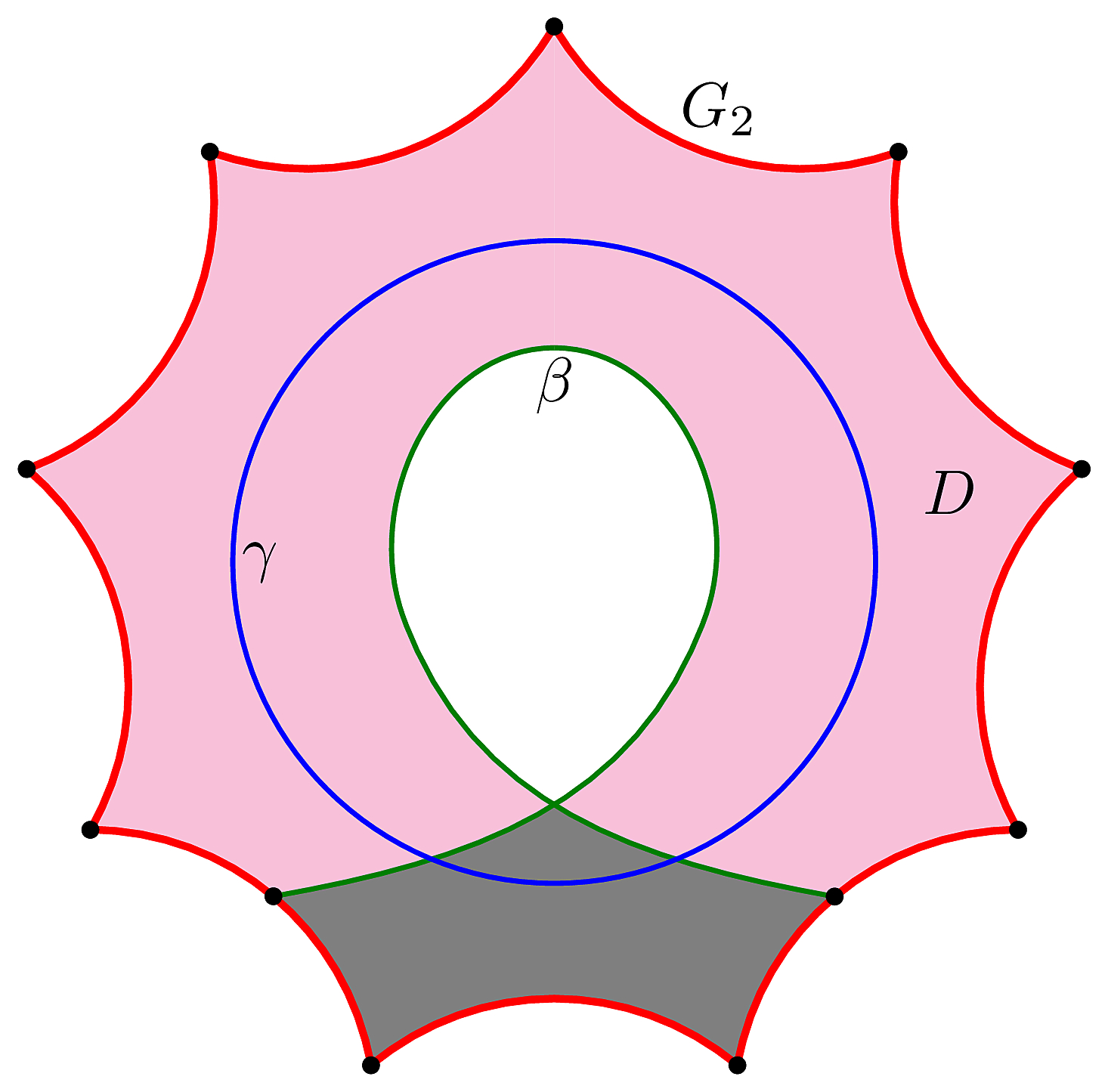}}
\hspace{20pt}
\subfloat{\includegraphics[width=0.32\textwidth]{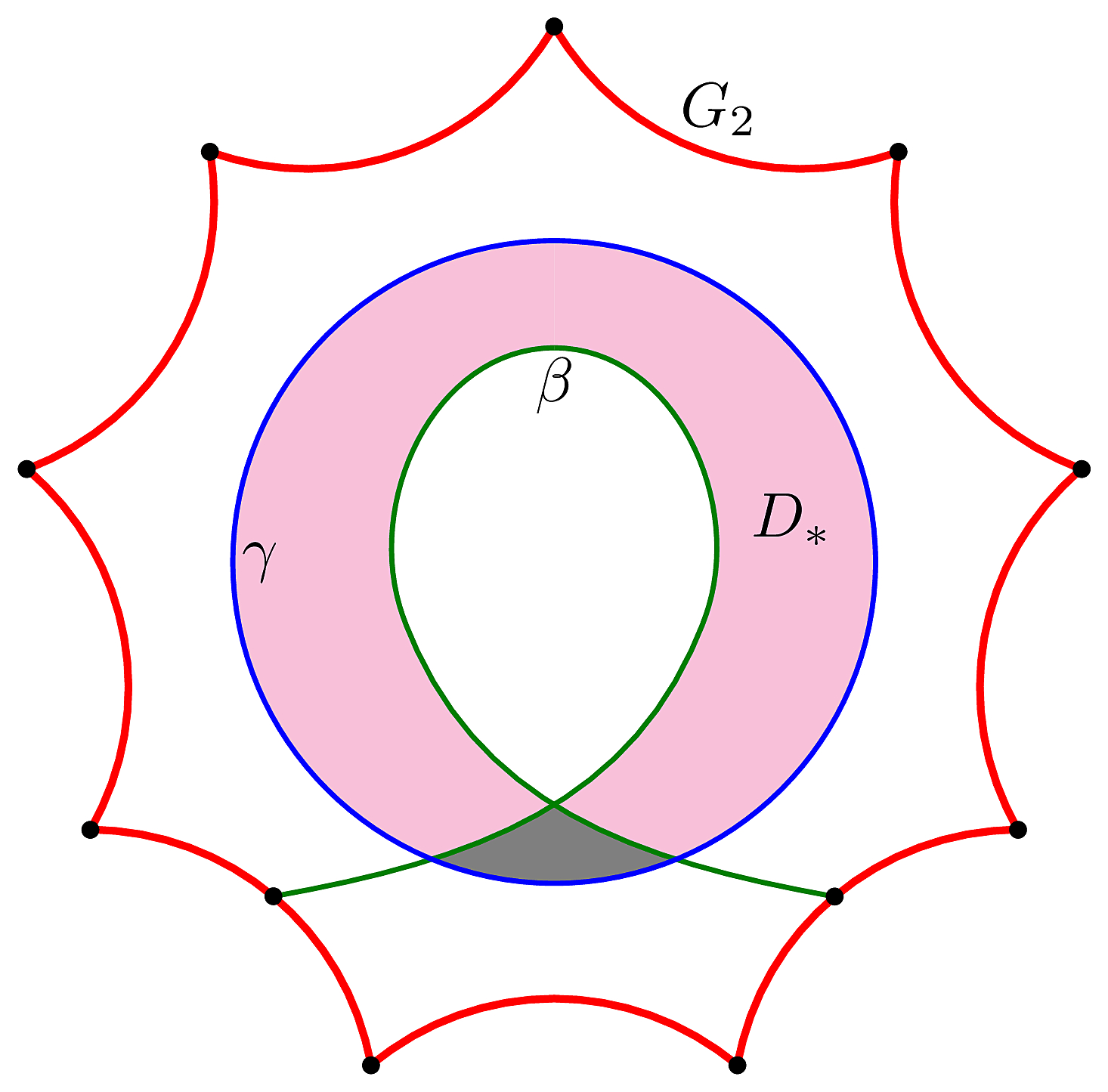}}
\caption{If $s_2$ and $G_2$ bound the immersed disk $D$ (left), then $s_2$ and $\gamma$ will bound the immersed bigon $D_*$ (right). The deeply shadowed regions are used twice.}
\label{fig:s2_not_boundary_parallel_2}
\end{figure}
if $\beta$ is $\partial$-parallel, then $\beta$ cobounds an (immersed) disk $D$ with $G_2$ and hence $\beta$ and $\gamma$ cobound an (immersed) bigon $D_*$ since $\beta\cap\gamma\neq\emptyset$ and $\gamma\cap G_2=\emptyset$. (This is more clear if viewed in the universal cover). It follows that any $\partial$-parallel proper arc for $(\Theta_2,G_2)$ is disjoint from $\gamma$ and $\beta$ is not $\partial$-parallel. Claim \ref{claim:no_essential_cutting_implies_filling} follows.

Since $\Gamma$ may have multiple intersection points, the endpoints of $\beta$ can be the vertices of $G_2$, and the two endpoints of $\beta$ can coincide as a proper arc in $X\setminus G_2$.
Let $\beta^\prime$ be the shortest proper geodesic arc in $X\setminus G_2$ that is homotopic to $\beta$ relative to $\partial(X\setminus G_2)$. The existence of $\beta^\prime$ is guaranteed by the Arzela-Ascoli theorem. $\beta^\prime$ is not necessarily unique. We let $\beta^\prime=\beta$ if $\ell(\beta)=\ell(\beta^\prime)$. Since $\beta^\prime$ is shortest, $\beta^\prime$ is perpendicular to $G_2$. Thus the two possible anomalies for $\beta$ can be avoided: the two endpoints of $\beta^\prime$ are distinct when viewed as a proper arc in $X\setminus G_2$, and disjoint from the vertices of $G_2$.
$\Theta_2\setminus G_2$ consists of proper geodesic arcs and closed geodesics in $X\setminus G_2$. Since $G_2$ is convex, every component of $X\setminus(\Theta_2\setminus \beta)$ is convex.

Let $\Theta_3=(\Theta_2\setminus\beta)\cup \beta^\prime$. Since $\ell\bracket{\beta^\prime}\leqslant\ell\bracket{\beta}$, we have $\ell(\Theta_3)\leqslant\ell(\Theta_2)$.

We show that $\Theta_3$ is filling. Suppose that $\Theta_3$ is not filling, then there exists an essential simple closed curve $\gamma$ in $X$ so that $\gamma\cap\Theta_3=\emptyset$. Since $\Theta_2$ is filling, $\gamma\cap\Theta_2\neq\emptyset$ and hence $\gamma\cap\beta\neq\emptyset$. For the universal covering $p:\bbd\to X$, a lift $\widetilde{\beta}$ in $\bbd$ of $\beta$ and some lift $\widetilde{\beta^\prime}$ in $\bbd$ of $\beta^\prime$ look like Figure \ref{figure:essential_arc_homotopy} $(A)$ or $(B)$, where $\Xi=p^{-1}\bracket{\Theta_2\setminus\beta}=p^{-1}\bracket{\Theta_3\setminus\beta^\prime}$.
\begin{figure}[htbp]
\centering
\includegraphics[width=0.82\textwidth]{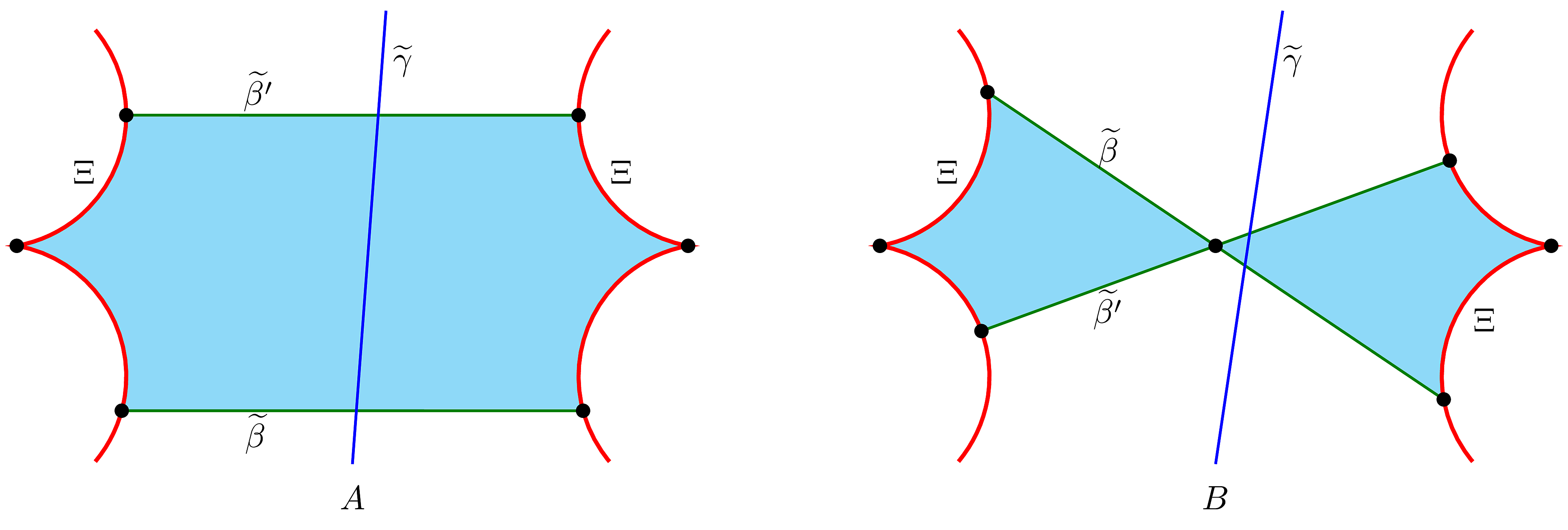}
\caption{The intersection of $\beta$ and $\gamma$ in the universal cover.}
\label{figure:essential_arc_homotopy}
\end{figure}
Since $\gamma\cap\beta\neq\emptyset$, for some lift $\widetilde{\gamma}$ of $\gamma$, we have $\widetilde{\gamma}\cap\widetilde{\beta}\neq\emptyset$ and then $\widetilde{\gamma}\cap\widetilde{\beta^\prime}\neq\emptyset$. Hence $\gamma\cap\beta^\prime\neq\emptyset$ and $\gamma\cap\Theta_3\neq\emptyset$. Therefore $\Theta_3$ is filling.

Starting from one end, the proper geodesic arc $\beta^\prime$ gives a cutting curve $s_2$.

We show that $s_2$ is essential. 
If $\beta^\prime$ is simple, then $s_2=\beta^\prime$ is of type V and essential. If $\beta^\prime$ is not simple, then $s_2$ is of type VI. If $s_2$ is not essential, then $s_2$ are $\partial$-parallel as in Figure \ref{figure:s2_rho_boundary_parallel_A}.
\begin{figure}[htbp]
\centering
\includegraphics[width=.28\textwidth]{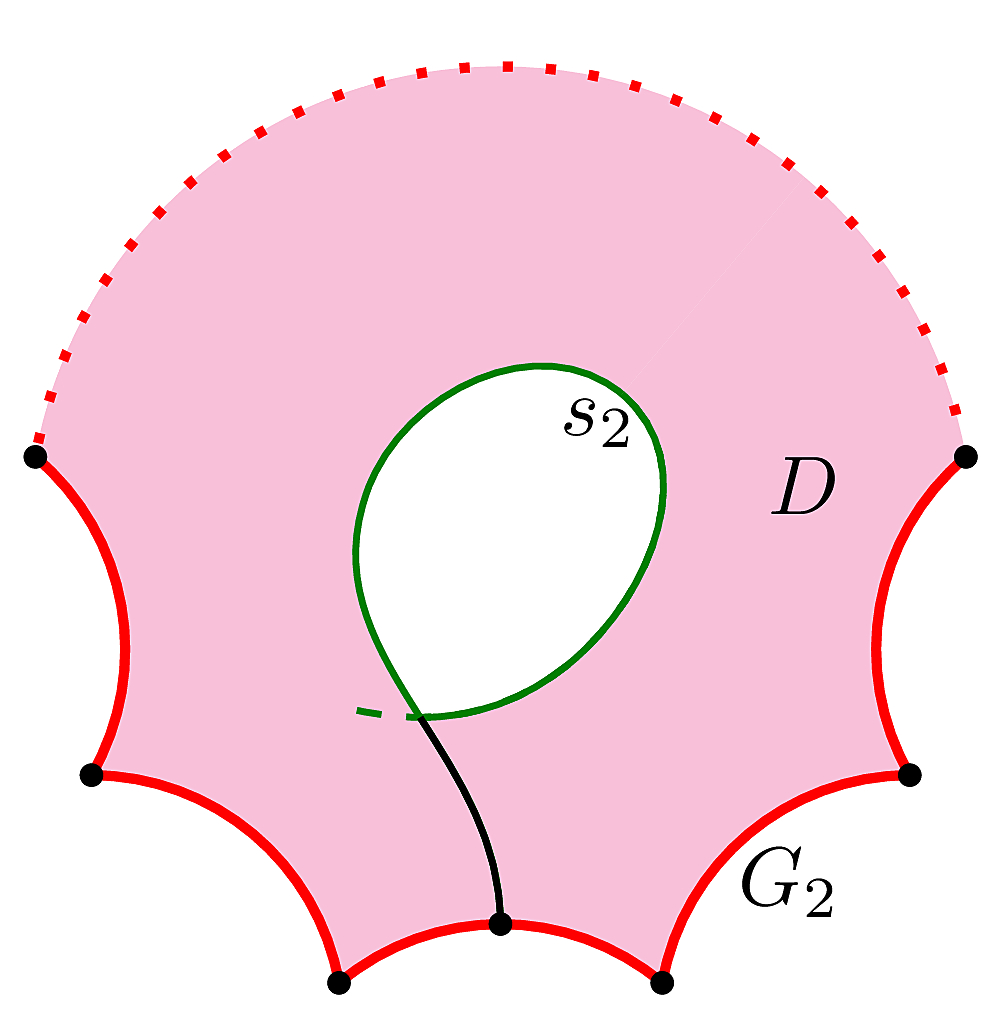}
\caption{A $\partial$-parallel cutting curve of type VI. The green part of $s_2$ will be referred to as the ``circle" part.}
\label{figure:s2_rho_boundary_parallel_A}
\end{figure}
$s_2$ is part of the essential geodesic proper arc $\beta^\prime$ for $X\setminus G_2$. 
As in Figure \ref{figure:s2_rho_boundary_parallel_B},
\begin{figure}[htbp]
\centering
\includegraphics[width=.9\textwidth]{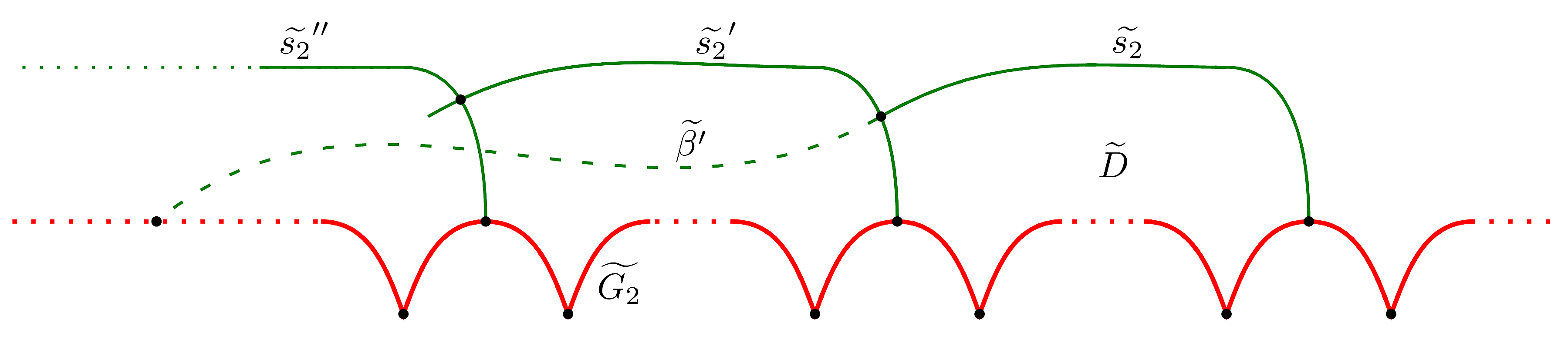}
\caption{A $\partial$-parallel cutting curve in $\bbd$.}
\label{figure:s2_rho_boundary_parallel_B}
\end{figure}
in the universal cover $\bbd$, a lift $\widetilde{s_2}$ of $s_2$ is the starting part of a lift $\widetilde{\beta^\prime}$ of $\beta^\prime$. The rest part of $\widetilde{\beta^\prime}$ will not intersect with any lift of the ``circle" part of $s_2$ since there will be a bigon between different lifts of $\beta^\prime$. Thus $\widetilde{\beta^\prime}$ and $\widetilde{G_2}$ will cobound a disk and $\beta^\prime$ and $G_2$ cobound an immersed disk, which contradicts to the essencity of $\beta^\prime$. Hence $s_2$ is essential.

Let $G_3=G_2\cup s_2$. Then $G_3$ is a convex subgraph of $\Theta_3$.

If $G_3$ is filling, let $G=G_3$ and the expected filling graph is constructed; otherwise we repeat the above argument for $(\Theta_3,G_3)$ to get $(\Theta_4, G_4)$, and etc. Since $\Gamma$ is a finite filling graph, the repeating process will eventually stop to get a filling graph $G_n$. Let $G=G_n$ and suppose that $X\setminus G$ consists of polygons $D_1, D_2, ..., D_k$.

It remains to show $G$ satisfies \eqref{eqn:subgraph_length_decreasing}, \eqref{eqn:subgraph_triangle_free} and \eqref{eqn:subgraph_complexity}. \eqref{eqn:subgraph_length_decreasing} holds since 
$$\ell(G_n)\leqslant \ell(\Theta_n)\leqslant\ell(\Theta_{n-1})\leqslant\cdots\leqslant\ell(\Theta_1)=\ell(\Gamma).$$

To show \eqref{eqn:subgraph_triangle_free}, observe that a polygon is produced only when an essential cutting curve is added to $G$. A cutting curve of type I, II, III or IV is added only once and no polygon is produced. For an essential cutting curve $s$ of type VI, since $s$ does not bound a disk with $G_2$ and the ``circle" part is not contractible, no polygon is produced. Hence a polygon is produced only when an essential cutting curve $s$ of type V is added. In this situation, since $s$ is essential, a polygon is produced only when $s$ cuts an annulus to get a polygon. The annulus cannot be bounded by two simple closed geodesics and it has at least 1 vertex. The cutting curve produces $4$ new vertices, and the new polygon has at least $5$ vertices. Hence every polygon $D_i$ has at least $5$ edges and \eqref{eqn:subgraph_triangle_free} holds.

The filling graph $G$ is obtained by adding essential cutting curves one by one. Two pairs of complementary inner angles are obtained when adding an essential cutting curve of II, III, IV, V, VI, and it follows that the average of the inner angles of $D_1, D_2, ..., D_k$ is $\frac\pi2$. Suppose that $D_1,\cdots,D_k$ are $m_1$-,...,$m_k$-gons respectively and the inner angles of $D_i$ are $A^i_1,\cdots,A^i_{m_i}$ ($i=1,\cdots,k$). Then the area of $D_i$ is $\area(D_i)=\sum_{j=1}^{m_i}(\pi-A^i_j)-2\pi$. 
Since the average of the inner angles is $\pi/2$, we have
$$4\pi(g-1)=\area(X)=\sum_{k=1}^k\area(D_i)=\sum_{i=1}^k\sum_{j=1}^{m_i}(\pi-A^i_j)-2k\pi=\sum_{i=1}^k \bracket{\frac{m_i\pi}2-2\pi}$$
and \eqref{eqn:subgraph_complexity} follows.
\end{proof}

\section{An isoperimetric inequality}\label{sec:iso}

In this section, we provide an isoperimetric inequality, which is a generalization of Sanki-Vadnere's \cite[Theorem 1.1]{sanki2021conjecture}. 

\begin{thm}\label{thm:isoperim}
Let $D_1,...,D_k$ be hyperbolic $m_i$-gons with $m_i\geqslant 4$. Let $D$ be a regular $m$-gon with inner angle $\theta\geqslant\frac\pi2$. If 
\begin{enumerate}
\item\label{item:thm_isoperim_gauss} $m-4 = \sum_{i=1}^k m_i -4k$;
\item\label{item:thm_isoperim_area} $\area(D) = \sum_{i=1}^k \area(D_i)$, 
\end{enumerate}
then 
$$\perim(D)\leqslant \sum_{i=1}^k \perim(D_i).$$
The equality holds if and only if $k=1$ and $D_1$ is a regular $m$-gon with inner angle $\theta\geqslant\frac\pi2$.
\end{thm}
\subsection{Perimeter functions of regular hyperbolic polygons}
For a hyperbolic regular $n$-gon with inner angle $\theta$ and area $a$, its perimeter can be represented by 
\begin{eqnarray}
P_n(a) &=& 2n\arccosh\bracket{\frac{\cos \frac{\pi}{n}}{\sin \frac{(n-2)\pi-a}{2n}}},\text{ or } \\
\phi_{\theta}(n) &=& 2n\arccosh\bracket{\frac{\cos \frac{\pi}{n}}{\sin \frac{\theta}{2}}}. 
\end{eqnarray}
By Gauss-Bonnet theorem, area of the $n$-gon can be represented by $\theta$ and $n$:
$$	a = (\pi-\theta)n-2\pi.$$
The derivative of $P_n(x)$ to $x$ is 
\begin{equation}\label{eqn:P_derivative_to_x}
P^\prime_n(x)=\frac{\dps{\cos\frac{\pi}n\tan\frac{2\pi+x}{2n}}}{\sqrt{\dps{\cos^2\frac{\pi}n-\cos^2\frac{2\pi+x}{2n}}}}.
\end{equation}

Let
$$f(n, a, x)= P_4(x)+P_n(a-x)-P_n(a).$$
We will show that the partial derivative $\frac{\partial f}{\partial x}(n,a,x)$ is positive and $n>10$ and $3\pi<a<\bracket{\frac{n}2-2}\pi$.

From \eqref{eqn:P_derivative_to_x}, we obtain
\begin{equation}\label{eqn:f_derivative_to_x}
\frac{\partial f}{\partial x}(n,a,x)=\frac{\tan\dps{\frac{2\pi+x}{8}}}{\dps{\sqrt{\sin\frac{x}4}}} - \frac{\dps{\cos\frac{\pi}n\tan\frac{2\pi+a-x}{2n}}}{\sqrt{\dps{\cos^2\frac{\pi}n-\cos^2\frac{2\pi+a-x}{2n}}}}.
\end{equation}
Under our assumption, all terms in \eqref{eqn:f_derivative_to_x} are non-negative. $\dps{\frac{\partial f}{\partial x}(n,a,x)>0}$ if and only if
\begin{equation}\label{eqn:f_derivative_to_x_positivity_1}
\frac{\tan^2\dps{\frac{2\pi+x}{8}}}{\sin\dps{\frac{x}4}}>\frac{\dps{\cos^2\frac{\pi}n\tan^2\frac{2\pi+a-x}{2n}}}{\dps{\cos^2\frac{\pi}n-\cos^2\frac{2\pi+a-x}{2n}}}.
\end{equation}
\eqref{eqn:f_derivative_to_x_positivity_1} is equivalent to
\begin{equation}\label{eqn:f_derivative_to_x_positivity_3}
\frac{\dps{\bracket{1+\sin\frac{x}4}\bracket{\cos\frac{2\pi}{n}-\cos\frac{2\pi+a-x}{n}}\bracket{1+\cos\frac{2\pi+a-x}{n}}}}{\dps{\sin\frac{x}4\bracket{1-\sin\frac{x}4}\bracket{1+\cos\frac{2\pi}n}\bracket{1-\cos\frac{2\pi+a-x}{n}}}}>1.
\end{equation}

\begin{lemma}\label{lemma:f_derivative_to_x_positive_n_bigger_than_8}
For $n\geqslant8$, $\dps{\frac{\partial f}{\partial x}(n,a,x)>0}$ when $\frac32\pi\leqslant a \leqslant\bracket{\frac{n}2-2}\pi$ and $0<x<\min(a-\pi,2\pi)$.
\end{lemma}

\begin{proof}
We prove the following equivalent form of \eqref{eqn:f_derivative_to_x_positivity_3}. Note that we have $a-x\geqslant\pi$.
It is easy to get $\dps{\frac{1+t}{t(1-t)}\geqslant 3+2\sqrt{2}}$ when $0<t<1$. Since $a-x\geqslant\pi$, we have
$$\frac{3\pi}{n}\leqslant\frac{2\pi+a-x}{n}<\frac{2\pi+\bracket{\frac{n}2-2}\pi}{n}=\frac{\pi}2.$$
From the basic facts
$$\sin3t\leqslant 3\sin t\text{ for }0<t<\frac{\pi}2 \qquad\text{and}\qquad
\sin(t)\geqslant\frac45t\ \text{when}\ 0\leqslant t\leqslant\frac{5}{16}\pi.
$$
we have
\begin{equation*}
\frac{\dps{\cos\frac{2\pi}{n}-\cos\frac{2\pi+a-x}{n}}}{\dps{1-\cos\frac{2\pi+a-x}{n}}}\geqslant \frac{\dps{\cos\frac{2\pi}{n}-\cos\frac{3\pi}{n}}}{\dps{1-\cos\frac{3\pi}n}}
=\frac{\dps{\sin\frac{5\pi}{2n}\sin\frac{\pi}{2n}}}{\dps{\sin^2\frac{3\pi}{2n}}}
\geqslant \frac13\cdot\frac{\dps{\sin\frac{5\pi}{2n}}}{\dps{\sin\frac{3\pi}{2n}}}
\geqslant \frac13\cdot \frac{\dps{\frac45\cdot\frac{5\pi}{2n}}}{\dps{\frac{3\pi}{2n}}}=\frac{4}{9}.
\end{equation*}
Therefore
$$\text{LHS of \eqref{eqn:f_derivative_to_x_positivity_3}}\geqslant (3+2\sqrt{2})\cdot \frac{4}{9}\cdot\frac{1}{2}\sim 1.29521>1.$$
\end{proof}

The following lemma is a generalization of \cite[Lemma 6.3]{sanki2021conjecture}. 

\begin{lemma}\label{lemma:Pn_second_derivative}
	When $n \geqslant 4$, $P''_n(x)$ has a unique zero in $(0,n\pi-2\pi)$ and is negative in the left of this zero, positive in the right of this zero. 
	\label{lem:convex_derivative}
\end{lemma}

\begin{proof}
From \eqref{eqn:P_derivative_to_x}, we get
\begin{equation}
	P_n''(x) = \frac{\cos \frac{\pi}{n} \sin^{-2} \frac{(n-2) \pi -x}{2n}}{2n \sqrt{\cos^2 \frac{\pi}{n} - \sin^{2} \frac{(n-2) \pi -x}{2n}}}
	- \frac{\cos \frac{\pi}{n} \cos^{2} \frac{(n-2) \pi -x}{2n}}
	{2n \left(\cos^2 \frac{\pi}{n} - \sin^{2} \frac{(n-2) \pi -x}{2n}\right)^{ \frac{3}{2}}}.
\end{equation}
$P_n''(x) \geqslant 0$ if and only if
$$\frac{\cos \frac{\pi}{n} \sin^{-2} \frac{(n-2) \pi -x}{2n}}{2n \sqrt{\cos^2 \frac{\pi}{n} - \sin^{2} \frac{(n-2) \pi -x}{2n}}} \geqslant \frac{\cos \frac{\pi}{n} \cos^{2} \frac{(n-2) \pi -x}{2n}}{2n \left(\cos^2 \frac{\pi}{n} - \sin^{2} \frac{(n-2) \pi -x}{2n}\right)^{ \frac{3}{2}} },$$
$$\cos^2 \frac{\pi}{n} - \sin^{2} \frac{(n-2) \pi -x}{2n} \geqslant \sin^2\frac{(n-2) \pi -x}{2n} \cos^2\frac{(n-2) \pi -x}{2n},$$
and equivalently,
$$\cos^2 \frac{\pi}{n} \geqslant \sin^{2} \frac{(n-2) \pi -x}{2n} + \sin^2\frac{(n-2) \pi -x}{2n} \cos^2\frac{(n-2) \pi -x}{2n}.$$

Since $t\mapsto \sin^2(t) + \sin^2(t)\cos^2(t)$ is injective when  $t\in (0,\pi/2)$, then $P_n''(x)$ has at most one zero in $(0,n\pi-2\pi)$. Since $\dps{\lim_{x\to 0+}P_n'(x) = +\infty}$ and $\dps{\lim_{x\to (n\pi-2\pi)-}P_n'(x) = +\infty}$ by \eqref{eqn:P_derivative_to_x}, $P_n''(x)$ has exactly one zero in $(0,n\pi-2\pi)$ and the lemma follows. 

\end{proof}

\begin{lemma}\label{lemma:Pn(x)_by_x_decreasing}
For $n=7,8,9,10$, $\dps{\frac{P_n(x)}x}$ is decreasing on $\dps{\bracket{0,\bracket{\frac{n}2-2}\pi}}$.
\end{lemma}

\begin{proof}
Since 
$$\frac{\dd{}}{\dd{x}}\bracket{\frac{P_n(x)}x}=\frac{xP_n^\prime(x)-P_n(x)}{x^2},$$
it suffices to show $xP_n^\prime(x)-P_n(x)<0$. The derivative of $xP_n^\prime(x)-P_n(x)$ is $xP^{\prime\prime}_n(x)$. By Lemma \ref{lemma:Pn_second_derivative}, $xP_n^\prime(x)-P_n(x)$ reaches its 
maximum at $0$ or $\bracket{\frac{n}2-2}\pi$. We have
$$\lim_{x\to0+}(xP^{\prime}_n(x)-P_n(x))=\lim_{x\to0+}xP^{\prime}_n(x)=0$$
It remains to show that 
\begin{equation*}
P_n\bracket{\bracket{\frac n2-2}\pi}> P_n^\prime\bracket{\bracket{\frac n2-2}\pi} \cdot \bracket{\frac n2-2}\pi,
\end{equation*}
which can be simplified as (by \eqref{eqn:P_derivative_to_x})
\begin{equation}\label{eqn:Pn(x)_by_x_decreasing_1}
2n\arccosh \bracket{\sqrt{2} \cos \frac{\pi}{n}}>\bracket{\frac n2-2}\pi\cdot \frac{\dps{\cos\frac{\pi}n}}{\sqrt{\dps{\cos^2\frac{\pi}n-\frac12}}}=\bracket{\frac n2-2}\pi\cdot\frac{\dps{\sqrt{2}\cos\frac{\pi}n}}{\dps{\sqrt{\cos\frac{2\pi}n}}}.
\end{equation}
For $n=7,8,9,10$, we have 
$$\frac32>\sqrt{2}\cos\frac{\pi}{n}>\sqrt{2}\cos\frac{\pi}{7}>\frac54,$$
hence, combining the basic fact
$$\arccosh(x)>0.55 x\quad\text{for}\quad\frac54<x<\frac32,$$
we have
$$2n\arccosh \bracket{\sqrt{2} \cos \frac{\pi}{n}}>1.1 \cdot n\cdot\sqrt{2}\cos\frac{\pi}n,$$
Then \eqref{eqn:Pn(x)_by_x_decreasing_1} reduces to the basic fact
$$1.21\cos\frac{2\pi}n-\bracket{\frac\pi2-\frac{2\pi}n}^2>0\quad\text{for}\quad n=7,8,9,10.$$
\end{proof}

We note that Lemma \ref{lemma:Pn(x)_by_x_decreasing} does not hold for large $n$, say $n>20$.

We will need the following result from \cite{sanki2021conjecture}.

\begin{prop}[Sanki-Vadnere]\label{prop:prep}
$\phi_{\theta}(n)$ and $P_n(x)$ satisfy the following properties:
\begin{enumerate}
\item\label{prop:phimono}
If we treat $\phi_{\theta}(n)$ as a function of $n$ defined on $\mathbb{R}_+$, then $\phi_{\theta}(n)$ is positive and strictly concave when $n\in (\frac{2\pi}{\pi-\theta}, +\infty)$ and $\phi_\theta(\frac{2\pi}{\pi-\theta})=0$. Moreover, a hyperbolic regular polygon with inner angle $\theta$ contains at least $\left\lfloor \frac{2\pi}{\pi-\theta} \right\rfloor +1 $ edges. (\cite[Lemma 7.1]{sanki2021conjecture}) 
\item\label{prop:pmono}
For $a\in \mathbb{R}_+$ and $n\in ( \frac{a}{\pi}+2,+\infty)$, $P_n(a)$ is strictly decreasing in $n$. (\cite[Lemma 5.1]{sanki2021conjecture})
\item\label{prop:pconcave}
$P_n(x) -P_{n+1}(x)\geqslant 0$, when $x\geqslant 0$; $P_n(x) -P_{n+1}(x)$ is monotonically increasing in $x$. (\cite[Proposition 4.1]{sanki2021conjecture})
\end{enumerate}
\end{prop}

\subsection{Isoperimetric inequality for two polygons}

We first prove the following proposition.

\begin{prop}\label{lem:four_edge}
For $n\geqslant 5$, $0 \leqslant a \leqslant  \bracket{\frac{n}{2}- 2}\pi$ and $0 \leqslant x\leqslant \min(a,2\pi)$, we have $f(n,a,x)\geqslant0$, that is,
\begin{equation}\label{eqn:four_edge}
P_4(x)+P_n(a-x)\geqslant P_n(a).
\end{equation}
The equality of \eqref{eqn:four_edge} holds if and only if $x=0$. 
\end{prop}

We make the following induction on $n$.

\begin{lemma}\label{cor:inequality_induction}
For $n\geqslant5$, if $a\in\gauss{0, \bracket{\frac{n}{2}- 2}\pi}$ and $x\in \gauss{0, \min(a,2\pi)}$ satisfies
$$P_4(x)+P_n(a-x)\geqslant P_n(a),$$
then 
$$P_4(x)+P_{n+1}(a-x)\geqslant P_{n+1}(a).$$
\end{lemma}

\begin{proof}
\begin{equation*}
\begin{aligned}
 & P_4(x)+P_{n+1}(a-x) \\
=& P_4(x) + P_n(a-x) + (P_{n+1}(a-x) -P_n(a-x) )\\
\ge& P_4(x) + P_n(a-x) + (P_{n+1}(a) -P_n(a) )  & \bracket{\text{By Proposition \ref{prop:prep}} (\ref{prop:pconcave})}\\
\ge& P_4(0) + P_n(a)+ (P_{n+1}(a) -P_n(a) ) & \bracket{\text{By induction}}\\ 
=& P_{n+1}(a). 
\end{aligned}
\end{equation*}
\end{proof}

\begin{proof}[Proof of Proposition \ref{lem:four_edge}]
By Proposition \ref{prop:prep}, $P_4(x)+P_m(a-x)\geqslant P_m(x)+P_m(a-x)$. 
$P_n(x)$ satisfies
\begin{equation*}
(P^\prime_n(x))^2=\frac{\dps{\cos^2\frac{\pi}n\tan^2\frac{2\pi+x}{2n}}}{\dps{\cos^2\frac{\pi}n-\cos^2\frac{2\pi+x}{2n}}}=\frac{\dps{\cos^2\frac{\pi}{n}\cdot\bracket{1-\cos^2\frac{2\pi+x}{2n}}}}{\dps{\cos^2\frac{2\pi+x}{2n}\cdot\bracket{\cos^2\frac{\pi}n-\cos^2\frac{2\pi+x}{2n}}}}.
\end{equation*}
Let $y=\cos^2\frac{2\pi+x}{2n}$ and $\kappa=\cos^2\frac{\pi}{n}$, then $y\in\left[\frac12,\kappa\right)$. The function $\dps{\frac{\kappa(1-t)}{t(\kappa-t)}}$ is strictly increasing on $(1-\sqrt{1-\kappa}, \kappa)$. It follows that $P_n^\prime(x)$ is decreasing and $P_n^{\prime\prime}(x)<0$ on $\left(0, \bracket{\frac{n}2-2}\pi\right)$ and  $P_n(x)$ is strictly concave on $\left[0, \bracket{\frac{n}2-2}\pi\right]$ when $n=5,6$. We remark that $1-\sqrt{1-\kappa}\leqslant\frac12$ requires $n\leqslant6$ and the concavity does not work for $n>6$. Therefore, for $n=5,6$ we have
$$P_n(x)+P_n(a-x)\geqslant P_n(0)+P_n(a) = P_n(a).$$ 

For $n=7,8,9,10$, by Lemma \ref{lemma:Pn(x)_by_x_decreasing}, we have
$$\frac{P_n(x)}{x}\geqslant \frac{P_n(a)}{a},\quad\text{and} \quad \frac{P_n(a-x)}{a-x}\geqslant \frac{P_n(a)}{a}.$$
Hence
$$P_n(x)+P_n(a-x)\geqslant \frac{x}{a}\cdot P_n(a)+\frac{a-x}{a}\cdot P_n(a)=P_n(a),$$
and \eqref{eqn:four_edge} follows from $P_4(x)\geqslant P_n(x)$ for $n>4$.

For $n>10$, when $a\leqslant3\pi$, the conclusion comes from the result for $n=10$ and induction on $n$ by Lemma \ref{cor:inequality_induction}. When $a\geqslant 3\pi$, the lemma follows from $f(n,a,0)\equiv0$ and $\frac{\partial f}{\partial x}(n,a,x)>0$ when $0<x<2\pi$ (Lemmas \ref{lemma:f_derivative_to_x_positive_n_bigger_than_8}).

In all of the three cases, the concerned functions are strictly increasing or strictly decreasing or strictly concave. Thus the equality holds only when $x=0$. 
\end{proof}

Now we are ready to prove the isoperimetric theorem when $k=2$. 

\begin{prop}
Let $D_1$, $D_2$ be regular hyperbolic $m_1$- and $m_2$-gons with positive areas. Let $D$ be a regular hyperbolic $m$-gon with inner angle $\theta \geqslant \frac{\pi}{2}$. If $m+4 = m_1+m_2$ and $\area(D)=\area(D_1)+\area(D_2)$, then 
$$\perim(D)<\perim(D_1) + \perim(D_2).$$

	\label{prop:k=2}
\end{prop}

\begin{proof}
Let $\theta_1$ and $\theta_2$ be the inner angles of $D_1$ and $D_2$ respectively.
By Gauss-Bonnet theorem, $\area(D)=\area(D_1)+\area(D_2)$ is equivalent to 
\begin{equation}\label{equ:area}
	(\pi-\theta)m - 2\pi = (\pi-\theta_1)m_1 - 2\pi + (\pi-\theta_2)(m+4-m_1) - 2\pi. 
\end{equation}

	Consider the function
	\begin{equation}
		F(m_1) = \phi_{\theta_1}(m_1)+\phi_{\theta_2}(m_2) = \phi_{\theta_1}(m_1)+\phi_{\theta_2}(m+4-m_1).
		\label{equ:phi}
	\end{equation}
	
	We treat $\theta_1, \theta_2, m$ as constants, then (\ref{equ:phi}) is a function in $m_1$. Our aim is to prove the minimum of this function in its domain is greater or equal to $\phi_\theta(m)$. 

	By Proposition \ref{prop:prep}, $\phi_{\theta_1}(m_1)+\phi_{\theta_2}(m+4-m_1)$ is defined on $\{m_1|m_1 \geqslant \max(4,\frac{2\pi}{\pi-\theta_1}), m+4-m_1 \geqslant \max(4,\frac{2\pi}{\pi-\theta_2}) \}$, namely $m_1\in [\max(4,\frac{2\pi}{\pi-\theta_1}),\min(m, m+4 - \frac{2\pi}{\pi-\theta_2})] $, and is strictly concave in this domain. 
	Therefore, the minimum of $\phi_{\theta_1}(m_1)+\phi_{\theta_2}(m+4-m_1)$ is obtained at the domain end points $\max(4,\frac{2\pi}{\pi-\theta_1})$ or $\min(m, m+4 - \frac{2\pi}{\pi-\theta_2})$. 

	First we consider the case $m_1 = \max(4,\frac{2\pi}{\pi-\theta_1})$. 

	If $\max(4,\frac{2\pi}{\pi-\theta_1}) = 4$, we consider $F(4) = \phi_{\theta_1}(4) + \phi_{\theta_2}(m)$. We let $x = area(D_1)$,  $a = area(D)$ and $a-x = area(D_2)$. Then $\phi_{\theta_1}(4) = P_4(x)$, $\phi_{\theta_2}(m) = P_m(a-x)$ and $\phi_\theta(m) = P_m(a)$. By Proposition \ref{lem:four_edge}, $P_4(x) + P_m(a-x) \geqslant P_m(a)$, namely $F(4) = \phi_{\theta_1}(4) + \phi_{\theta_2}(m) \geqslant \phi_\theta(m)$. 

	If $\max(4,\frac{2\pi}{\pi-\theta_1}) = \frac{2\pi}{\pi-\theta_1}$, we consider $F(\frac{2\pi}{\pi-\theta_1})$. Then $\phi_{\theta_1}(\frac{2\pi}{\pi-\theta_1}) = 0$, $F(\frac{2\pi}{\pi-\theta_1}) = 0 + \phi_{\theta_2}(m+4 - \frac{2\pi}{\pi-\theta_1})$. 
	By (\ref{equ:area}), let $m_1 = \frac{2\pi}{\pi-\theta_1}$, then 
	\[
	(\pi-\theta)m - 2\pi =   (\pi-\theta_2)\left(m+4-\frac{2\pi}{\pi-\theta_1}\right) - 2\pi. 
	\]
Let $a = 
(\pi-\theta)m - 2\pi =   (\pi-\theta_2)\left(m+4-\frac{2\pi}{\pi-\theta_1}\right) - 2\pi$. Then $\phi_\theta(m) = P_m(a)$ and   $F(\frac{2\pi}{\pi-\theta_1}) = 0 + \phi_{\theta_2}(m+4 - \frac{2\pi}{\pi-\theta_1}) = P_{m+4 - \frac{2\pi}{\pi-\theta_1}}(a)$. Since $\max(4,\frac{2\pi}{\pi-\theta_1}) = \frac{2\pi}{\pi-\theta_1}$, $m+4 - \frac{2\pi}{\pi-\theta_1} < m$. Therefore by Proposition \ref{prop:prep}(\ref{prop:pmono}), $ P_{m+4 - \frac{2\pi}{\pi-\theta_1}}(a) > P_m(a)$, namely $F(\frac{2\pi}{\pi-\theta_1}) > \phi_\theta(m)$. 

If $m_1 = \min(m, m+4 - \frac{2\pi}{\pi-\theta_2})$, the proof is exactly the same as the case $m_1 = \max(4,\frac{2\pi}{\pi-\theta_1})$. 
\end{proof}

\begin{cor}
Let $D$ be a regular hyperbolic $m$-gon with inner angle $\theta\geqslant\frac\pi2$, and $D_1$ and $D_2$ be hyperbolic $m_1$- and $m_2$-gons respectively, $m_1+m_2 = m+4$. If $area(D_1)+ area(D_2) = area(D)$, $perim(D_1) + perim(D_2) = perim(D)$, $m_1 \geqslant m_2$, then $D_1 =D$, $D_2$ is a degenerated polygon with area and perimeter $0$. 
	\label{cor:equal}
\end{cor}

\begin{proof}
	By the proof of Proposition \ref{prop:k=2}, if $\phi_{\theta_1}(m_1)+ \phi_{\theta_2}(m_2) = \phi_{\theta}(m)$, then $m_1=m$ and $m_2 =4$. By Proposition \ref{lem:four_edge}, if $P_4(x)+P_m(a-x) = P_m(a)$, then $x=0$ and this corollary follows. 
\end{proof}

\subsection{General case}
We prove the isoperimetric theorem in this subsection by induction. 

\begin{lemma}
	Let $D_i$ be regular hyperbolic $m_i$-gons with inner angle $\theta_i$, $i = 1,...,k$. $\theta_1 \geqslant \theta_2 \geqslant ... \geqslant \theta_k$. Let $D$ be a regular hyperbolic $m$-gon with inner angle $\theta \geqslant \frac{\pi}{2}$. If 
	\begin{enumerate}
		\item $\sum_{i=1}^{k}\area(D_i) =\area(D)$ ; \label{cond:area}
		\item $m -4 =\sum_{i=1}^{k} m_i -4k $. \label{cond:euler}
	\end{enumerate}
	Then 
	\begin{enumerate}
		\item $\theta_1 \geqslant \frac{\pi}{2}$ ;
		\item $\theta_k \leqslant \theta$. 
	\end{enumerate}
	\label{lem:induction}
\end{lemma}
\begin{proof}
	By Gauss-Bonnet theorem, $\area(D_i) = (\pi-\theta_i)m_i - 2\pi$. Therefore, the condition (\ref{cond:area}) is equivalent to 
	\[
		\pi\sum_{i=1}^{k} m_i - \sum_{i=1}^{k} m_i \theta_i - 2k\pi = (\pi-\theta)m - 2\pi. 
	\]
	Using the condition (\ref{cond:euler}), we get 
	\[
		(2k-2)\pi + m\theta = \sum_{i=1}^{k} m_i \theta_i. 
	\]
Then 
\begin{eqnarray*}
		(4k-4+m)\frac{\pi}{2} =	(2k-2)\pi + m \frac{\pi}{2} 
				      \leqslant (2k-2)\pi + m\theta 
				      = \sum_{i=1}^{k} m_i \theta_i 
				     \leqslant  \sum_{i=1}^{k} m_i \theta_1 = (m+4k-4)\theta_1.
\end{eqnarray*}
Hence $\theta_1 \geqslant \frac{\pi}{2}$. On the other hand
	\begin{eqnarray*}
		(4k-4+m)\theta =	(4k-4)\theta + m \theta 
				      \geqslant (2k-2)\pi + m\theta 
				      = \sum_{i=1}^{k} m_i \theta_i 
				     \geqslant  \sum_{i=1}^{k} m_i \theta_k = (m+4k-4)\theta_k.
	\end{eqnarray*}
	Hence $\theta_k \leqslant \theta$. 
\end{proof}
Now we are ready to prove the isoperimetric theorem
\begin{proof}[Proof of Theorem \ref{thm:isoperim}]
	Without loss of generality, we assume the inner angle of $D_i$ is $\theta_i$ and $\theta_1 \geqslant \theta_2 \geqslant ... \geqslant \theta_k$. We consider a series of hyperbolic regular polygons $ \widetilde{D}_1, \widetilde{D}_2,..., \widetilde{D}_k$. 
	$\widetilde{D}_j$ is defined as the polygon with $ \widetilde{m}_j:= \sum_{i=1}^{j} m_i -4j+4$ edges and area equals to $\sum_{i=1}^{j}\area(D_i)$. 
	\begin{lemma}
		The inner angle of $ \widetilde{D}_j $ (denoted as $ \widetilde{\theta_j}$) is greater or equal than $ \frac{\pi}{2}$. 
		\label{lem:angle}
	\end{lemma}
	\begin{proof}
		We prove it by induction. 
		We observe that $\area(\widetilde{D}_k) = \sum_{i=1}^{k}\area(D_i)=\area(D)$ and $ \widetilde{m}_k = \sum_{i=1}^{k} m_i -4k+4 = m$. Therefore $\widetilde{D}_k$ is isometric to  $D$, hence $ \widetilde{\theta}_k = \theta \geqslant \frac{\pi}{2}$. 

		Assume $ \widetilde{\theta}_j \geqslant \frac{\pi}{2}$, then it is sufficient to prove that $ \widetilde{\theta}_{j-1} \geqslant \frac{\pi}{2}$. Applying Lemma \ref{lem:induction} to $D_1,D_2,...,D_j$, we have $\theta_j \leqslant \widetilde{\theta}_j$. 
		Then we consider $\widetilde{D}_{j-1}$, $D_j$ and $\widetilde{D}_j$. 
		We mention that from the definition of $ \widetilde{m}_j$ and $\area(\widetilde{D}_j) = \area(\widetilde{D}_{j-1})+\area(D_{j})$, we have 
		\begin{equation}
			\widetilde{m}_j +4 = \widetilde{m}_{j-1}+ m_j ; \widetilde{m}_j\widetilde{\theta}_j + 2\pi = \widetilde{m}_{j-1}\widetilde{\theta}_{j-1} + m_j \theta_j. 
		\label{equ:twopolygon}	
		\end{equation}
		If $\widetilde{\theta}_{j-1}\geqslant \theta_j$, then by Lemma \ref{lem:induction}, $\widetilde{\theta}_{j-1}\geqslant \frac{\pi}{2}$. If $\widetilde{\theta}_{j-1}< \theta_j$, then $\widetilde{\theta}_{j-1}< \theta_j \leqslant  \widetilde{\theta}_j$. By \eqref{equ:twopolygon}, 
		\[
			\widetilde{m}_j\widetilde{\theta}_j + 2\pi = \widetilde{m}_{j-1}\widetilde{\theta}_{j-1}+ m_j \theta_j > (\widetilde{m}_{j-1}+ m_j) \widetilde{\theta}_j= (\widetilde{m}_j+4)\widetilde{\theta}_j. 
		\]
		Hence $ \widetilde{\theta}_j< \frac{\pi}{2}$, which contradicts to the assumption $ \widetilde{\theta_j} \geqslant \frac{\pi}{2}$. 
	\end{proof}
	The $k=2$ case has been proved in Proposition \ref{prop:k=2}. Assume this inequality is true for $k-1$ polygons. The remaining work is to prove the isoperimetric theorem for $k$ polygons $D_1,...,D_k$. By Lemma \ref{lem:angle}, the inner angle $ \widetilde{\theta}_2$ of $ \widetilde{D}_2$ is greater or equal than $ \frac{\pi}{2}$. Therefore by Proposition \ref{prop:k=2}, $$\perim( \widetilde{D}_2) \leqslant\perim(D_1)+\perim(D_2). $$
Then for the polygons $ \widetilde{D}_2, D_3,...,D_k$, it is easy to check that $$ \widetilde{m}_2 + \sum_{i=3}^k m_i - 4(k-1) = m -4.$$ Therefore by induction, 
$$\perim(D) \leqslant \perim( \widetilde{D}_2) + \sum_{i=3}^k \perim(D_i) \leqslant \sum_{i=1}^k \perim(D_i) .$$ 
If $\perim(D) = \sum_{i=1}^k\perim(D_i) $, then 
$$\perim(D) = \perim( \widetilde{D}_2) + \sum_{i=3}^k\perim(D_i) = \sum_{i=1}^k\perim(D_i) .$$ 
Therefore $$\perim( \widetilde{D}_2) = \perim(D_1) + \perim(D_2).$$ 
By Corollary \ref{cor:equal}, $\widetilde{D}_2$ equals $D_1$ or $D_2$ and the other polygon is a quadrilateral with $0$ area and perimeter. Without loss of generality, we assume $\widetilde{D}_2= D_1$. Thus 
$$\perim(D) =\perim( \widetilde{D}_2) + \sum_{i=3}^k \perim(D_i) = \perim(D_1) + \sum_{i=3}^k\perim(D_i).$$
By induction, $k=1$ and $D= D_1$. 

\end{proof}

The following example shows that Theorem \ref{thm:isoperim} does not hold if triangles are allowed.

\begin{example}\label{emp:triangle_fail}
Let $D_1$ be a regular hexagon with area $x$ and $D_2$ be a regular triangle with area $a-x$. Let $D$ be a regular pentagon with area $a$. Then the conditions \eqref{item:thm_isoperim_gauss} and \eqref{item:thm_isoperim_area} in Theorem \ref{thm:isoperim} are satisfied. If the isoperimetric inequality holds, then
\begin{equation}\label{eqn:example_triangle_fail}
P_{6} (x) + P_3(a-x) \geqslant P_5(a).
\end{equation}

However, by Proposition \ref{prop:prep}(\ref{prop:pmono}), we have
$$\lim_{x\to a-}\bracket{P_6(x)+P_3(a-x)}=P_6(a)<P_5(a)$$
Thus for $x$ sufficiently close to $a$, we have $P_{6} (x) + P_3(a-x) < P_5(a)$. In fact, a direct computation show that $P_6(4.99)+P_3(0.01)+\frac12<P_5(5)$, which fails \eqref{eqn:example_triangle_fail}.
\end{example}

\section{Shortest filling geodesics on hyperbolic surfaces}\label{sec:minimal}

In this section, we first construct a hyperbolic surface $X \in\mathcal{M}_g$ and a filling geodesic $\gamma \subset X$ so that $X\backslash\gamma$ is a right-angled regular hyperbolic $(8g-4)$-gon, which implies that the minimal length in Question 1 is at most $\frac12P_{8g-4}$ . We then prove the main results by combining results in previous sections.

For a genus $g$ surface $S_g$, we have a filling curve $\alpha \subset S_g$ as indicated in Figure \ref{fig:shortest_candidate}. 
\begin{figure}[htbp]
\centering
\includegraphics[width=.96\textwidth]{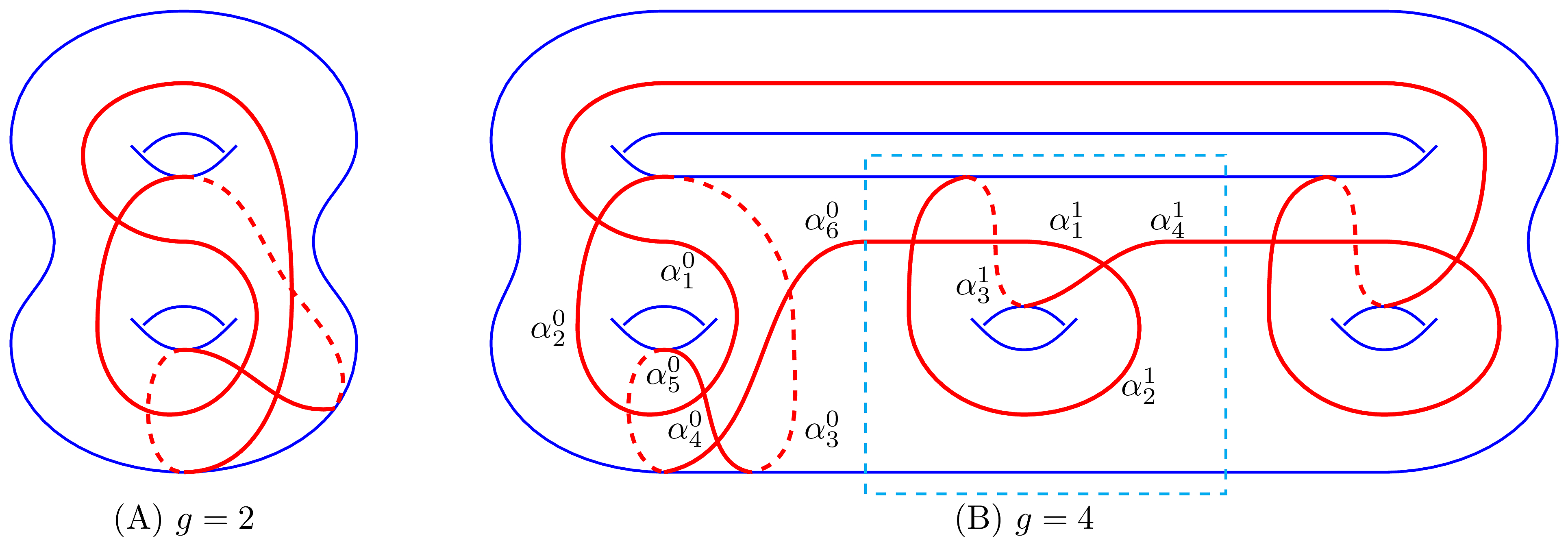}
\caption{A closed filling geodesic with $(2g-1)$ self-intersections. For $g>2$, the part in the dashed rectangle is the \emph{building block} and there are $g-3$ copies. The double points separate $\alpha$ into edges labelled along the curve as in (B). For conciseness, we only label some edges and the other edges are labelled accordingly.}
\label{fig:shortest_candidate}
\end{figure}
When $g=2$ (Figure \ref{fig:shortest_candidate}(A)), $\alpha$ has 3 double self-intersection points. When $g>2$ (Figure \ref{fig:shortest_candidate}(B)), the curve $\alpha$ have $g-3$ copies of the building block and has $2g-3$ double self-intersection points.  It is not hard to see that $\alpha$ is filling and $S_g\backslash\alpha$ is a $(8g-4)$-gon. Hence $S_g$ can be obtained from a $(8g-4)$-gon $Q$ with edges labelled clockwise as $\alpha^0_6,\alpha^0_3,\alpha^0_1,\alpha^0_4,\alpha^0_6, \alpha^1_3, \alpha^1_1, \alpha^1_2, \alpha^1_3, \alpha^1_1, \alpha^1_4,\cdots,\alpha^{g-2}_3,\alpha^{g-2}_1,\alpha^{g-2}_2,\alpha^{g-2}_3,\alpha^{g-2}_1, \alpha^{g-2}_4, \linebreak \alpha^0_3,\alpha^0_5, \alpha^0_1,\alpha^0_2,\alpha^0_5, 
\alpha^0_4, \alpha^0_2, \alpha^{g-2}_4,\alpha^{g-2}_2,\cdots,\alpha^1_4,\alpha^1_2$ by gluing the edges with the same labels (the labeling can be read from Figure \ref{fig:shortest_candidate}). The genus $2$ case is shown in Figure \ref{fig:genus_2_12_gon}. 
\begin{figure}[htbp]
\centering
\includegraphics[width=.4\textwidth]{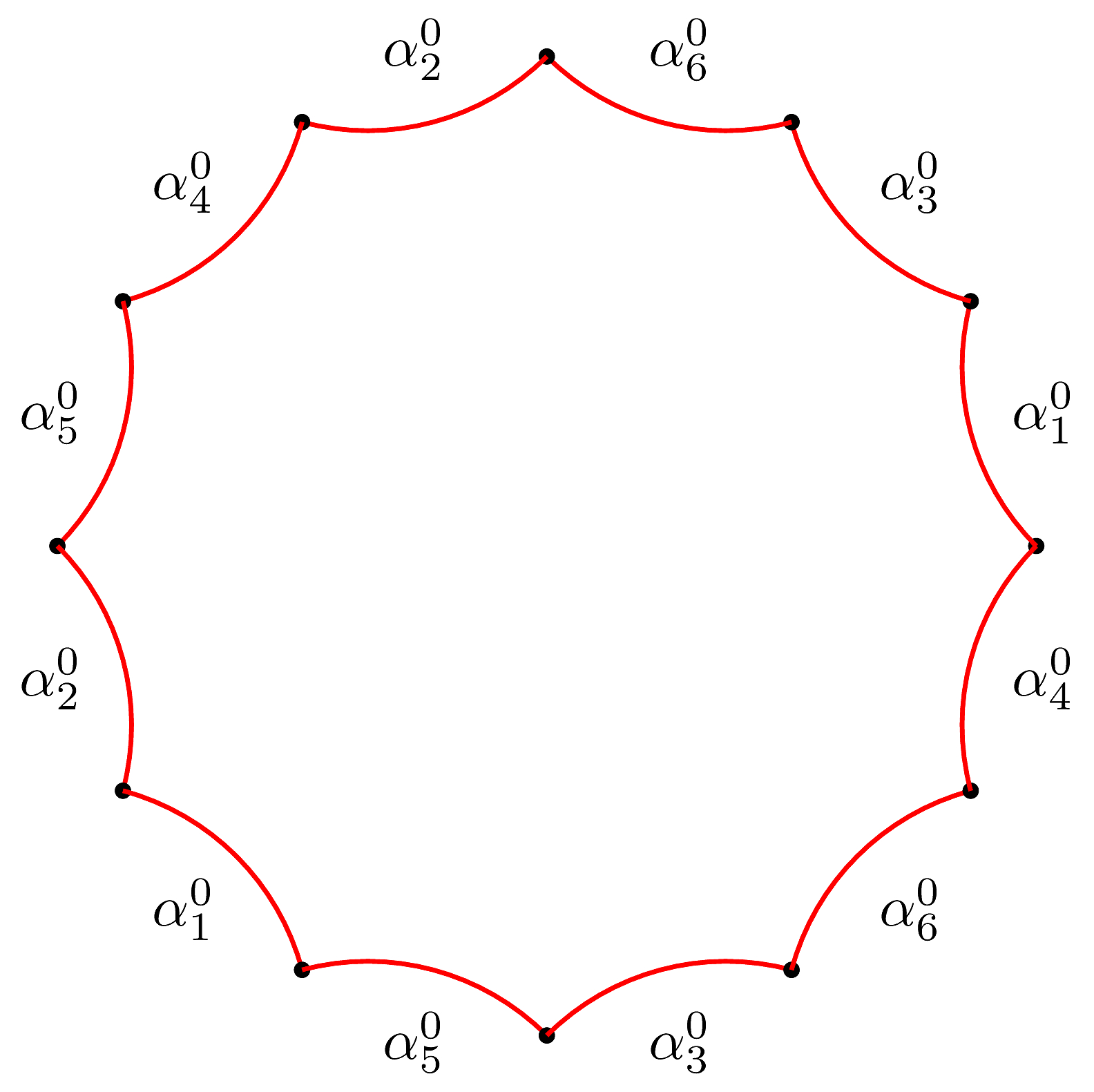}
\caption{Gluing pattern for genus 2.} \label{fig:genus_2_12_gon}
\end{figure}

Let $P$ be a regular hyperbolic $(8g-4)$-gon with right angles. By gluing $P$ in the same pattern as $Q$, we obtain a hyperbolic surface $X$ with genus $g$ and the boundary of the $P$ becomes a filling geodesic $\gamma$ as in Figure \ref{fig:shortest_candidate}. The length of $\gamma$ is half the perimeter of $P$.


\begin{proof}[Proof of Theorem \ref{thm:main}]
Let $X$ be a hyperbolic surface with genus $g$ and $\Gamma $ be a filling multi-geodesic on $X$. By Theorem \ref{thm:filling_graph}, there is a filling geodesic graph $G$ on $X$ such that 
\begin{enumerate}
	\item $\ell(G) \leqslant \ell(\Gamma)$. 
	\item $X\backslash G$ consists of polygons $D_1, D_2, ..., D_k$.  $D_i$ has $m_i$ edges, $m_i\geqslant5$. 
	\item $\sum_{i=1}^k (m_i-4) = 8g-8$. 
\end{enumerate}
Then, by Theorem \ref{thm:isoperim}, we have
$$\ell(\Gamma) \geqslant \ell(G) = \frac{1}{2} \sum_{i=1}^{k}\perim(D_i) \geqslant \frac{1}{2} P_{8g-4}.$$ 
The above construction provided a single filling geodesic which realized the length $\frac12P_{8g-4}$ for every genus $g$ and Theorem \ref{thm:main} follows.
\end{proof}

\begin{proof}[Proof of Theorem \ref{thm:rigid}]

Suppose that $\Gamma$ is a filling multi-geodesic on a genus $g$ hyperbolic surface $X$ with $\ell(\Gamma) = \frac{1}{2} P_{8g-4}$. By Theorem \ref{thm:filling_graph}, we get a filling graph $G$ so that $\ell(G)\leqslant\ell(\Gamma)$ and $X\setminus G$ consists of polygons with edge number $\geqslant5$. We have $\ell(G)\geqslant\frac12P_{8g-4}$ by Theorem \ref{thm:isoperim} and it follows that $\ell(G)=\ell(\Gamma)$. From the proof of Theorem \ref{thm:filling_graph}, all essential proper arcs are shortest under relative homotopy since we have $\ell(G)<\ell(\Gamma)$ otherwise. Hence $G$ is a subgraph of $\Gamma$ and then $G=\Gamma$ since they have the same length.

Now $X\setminus\Gamma$ consists of polygons with edge number $\geqslant5$ and $\ell(\Gamma) = \frac{1}{2} P_{8g-4}$, $X\setminus\Gamma$ is a regular right-angled $(8g-4)$-gon by the equality result in Theorem \ref{thm:isoperim}. 
\end{proof}



\end{document}